\documentclass[a4paper,11pt]{article}
\usepackage[centertags]{amsmath}
\usepackage{bm}
\usepackage{amsthm}
\usepackage{amssymb}
\usepackage[utf8]{inputenc}
\usepackage[english]{babel}
\usepackage{setspace}
\usepackage{tikz}

\usepackage{dsfont}

\newcommand{\ud}{{\mathrm{d}}}
\newcommand{\uD}{\mathrm{D}}

\newcommand{\base}{{{B}}}
\newcommand{\hauteur}{{{L}}}
\newcommand{\hauteursur}{{{\hauteur^+}}}
\newcommand{\hauteursous}{{{\hauteur^-}}}
\newcommand{\Interv}[1]{{{\mathcal{I}^{#1}}}}
\newcommand{\Interveta}[2]{{{\mathcal{I}_{#1}^{#2}}}}

\DeclareMathOperator{\Lapl}{\triangle}
\DeclareMathOperator{\Rot}{\bm{curl}}
\DeclareMathOperator{\Div}{{div}}

\def\R{\mathbb{R}}
\def\D{{\mathcal D}}
\def\curl{\mbox{curl }}
\def\div{\mbox{div }}
\def\HH{\mathbb{H}}
\def\ds{\longrightarrow}
\def\LL{\mathbb{L}}
\def\dd#1#2{\frac{\partial #1}{\partial #2}}
\def\dsp{\displaystyle}
\def\CC{{\mathcal C}}
\newcommand{\vect}{{\wedge}}
\newcommand{\cart}{{\times}}

\numberwithin{equation}{section}

\title
{Weak solutions to the Landau-Lifshitz-Maxwell system 
with nonlinear Neumann boundary conditions
arising from surface energies}
\author{Gilles Carbou\thanks{Laboratoire de Math\'ematiques et de leurs Applications de Pau, CNRS UMR 5142, Universit\'e de Pau et
    des Pays de l'Adour \texttt{gilles.carbou@univ-pau.fr}},  
Pierre Fabrie\thanks{Institut de Math\'ematiques de Bordeaux, CNRS UMR
  5251, Institut Polytechnique de Bordeaux
  \texttt{Pierre.Fabrie@math.u-bordeaux1.fr}},
 and Kévin Santugini\thanks{Institut Mathématiques de
Bordeaux, CNRS UMR5251,  MC2, INRIA Bordeaux - Sud-Ouest
\texttt{Kevin.Santugini@math.u-bordeaux1.fr}}}
\date{\today}

\newtheorem{theorem}{Theorem}
\newtheorem{lemma}[theorem]{Lemma}

\theoremstyle{definition}
\newtheorem{definition}[theorem]{Definition}

\theoremstyle{remark}

\begin{document}
\maketitle
\begin{abstract}
We study the Landau-Lifshitz system 
associated with Maxwell equations in a bilayered ferromagnetic body when
super-exchange and surface anisotropy interactions are present in the
spacer in-between the layers. In the presence of these surface
energies, the Neumann boundary condition becomes nonlinear.
We prove, in three dimensions, the existence of global weak solutions to the
Landau-Lifshitz-Maxwell system with nonlinear Neumann boundary conditions.
\end{abstract}

\section{Introduction}
Ferromagnetic materials are widely used in the industrial
world. Their four main applications are data storage (hard drives), furtivity,
communications (wave circulator), and energy (tranformers). For an
introduction to ferromagnetism, see~Aharoni\cite{Aharoni:introduc} or 
Brown\cite{Brown:microm}. 

The state of a ferromagnetic body is characterized by its
magnetization $\bm{m}$, a vector field whose norm is equal to $1$
inside the ferromagnetic body and null outside. The evolution of
$\bm{m}$ can be modeled by the Landau-Lifshitz equation
\begin{equation*}
\frac{\partial\bm{m}}{\partial t}=
-\bm{m}\vect\bm{h}_{\mathrm{tot}}
-\alpha\bm{m}\vect(\bm{m}\vect\bm{h}_{\mathrm{tot}}),
\end{equation*}
where $\bm{h}_{\mathrm{tot}}$ depends on $\bm{m}$ and contains various
contributions. In particular, in this paper, $\bm{h}_{\mathrm{tot}}$
includes various volumic and surfacic energies, among which 
the solution to Maxwell equations and several surfacic terms such as 
super-exchange and surface anisotropy. 

F.~Alouges and A.~Soyeur\cite{Alouges.Soyeur:OnGlobal} established 
the existence and the non-uniqueness of weak solutions to the Landau-Lifshitz system when only
exchange is present, \textit{i.e.} when
$\bm{h}_{\mathrm{tot}}=\Lapl\bm{m}$, see also A.Visintin~\cite{Visintin:Landau}. 
S.~Labbé \cite[Ch. 10]{labbe:mumag} extended the existence result  in the presence of the 
magnetostatic field. In the absence of the exchange interaction,
J.L.~Joly, G.~Métivier and J.~Rauch 
obtain global existence and uniqueness results in~\cite{Joly:global}.
G. Carbou and P. Fabrie \cite{Carbou.Fabrie:Time} proved
the existence of weak solutions when the Landau-Lifshitz equation is 
associated with Maxwell equations. 
K.~Santugini proved in~\cite{Santugini:SolutionsLL}, see
also~\cite[chap. 6]{Santugini:These}, the existence of
weak solutions globally in
time to the magnetostatic Landau-Lifshitz system in the presence of surface energies
that cause the Neumann boundary conditions to become nonlinear.
In this paper, we prove the existence of weak solutions to the full
Landau-Lifshitz-Maxwell system with the nonlinear Neumann boundary
conditions arising from the super-exchange and the surface anisotropy energies. In addition, we address  the long time behavior  by describing the $\omega$-limit set of the trajectories.

\vspace{3mm}

The plan of the paper is the following. In \S\ref{sect:notations}, we introduce several notations we use
throughout this paper. In \S\ref{sect:MicromagneticModel}, we recall the micromagnetic
model.  In \S\ref{sect:LandauLifshitzSystem}, we state our main
theorems. 
Theorem~\ref{theo:ExistenceWeakLandauLifshitzMaxwellSurfaceEnergies}
 states the global existence in time of weak solutions to the
Landau-Lifshitz system with the nonlinear Neumann Boundary conditions
arising from the super-exchange and the surface anisotropy energies. Theorem \ref{theo-omegalim} describes the $\omega$-limit set of a solution given by the previous theorem.
In \S\ref{sect:prerequisites}, before starting the proofs, we recall technical results on Sobolev
Spaces we use in this paper. We prove Theorem~\ref{theo:ExistenceWeakLandauLifshitzMaxwellSurfaceEnergies}
 in \S\ref{sect:proofExistenceWeakLLMSurfaceEnergies} and Theorem~\ref{theo-omegalim} in \S\ref{section-omegalimit}.

\vspace{2mm}

{\bf Notation} Throughout the paper, $\lVert\cdot\rVert$ denotes the euclidean
norm over $\mathbb{R}^d$ where $d$ is a positive integer, often equal
to $3$. When refering to the $\mathrm{L}^2$ norm over a measurable set
$A$, we use instead the $\lVert\cdot\rVert_{\mathrm{L}^2(A)}$ notation.

\section{Geometry of spacers and related notations}\label{sect:notations}

In this paper, we consider a ferromagnetic domain with spacer. We denote by 
$\Omega=\base\cart\Interv{}$ this domain, 
where $\base$ is a bounded domain of $\mathbb{R}^2$ with smooth
boundary and $\Interv{}$ is the interval $\rbrack-\hauteursous,\hauteursur\lbrack\setminus\{0\}$.
We set $Q_T= \rbrack0,T\lbrack\cart\Omega$ where $\hauteursur$ and $\hauteursous$
are two positive real numbers.

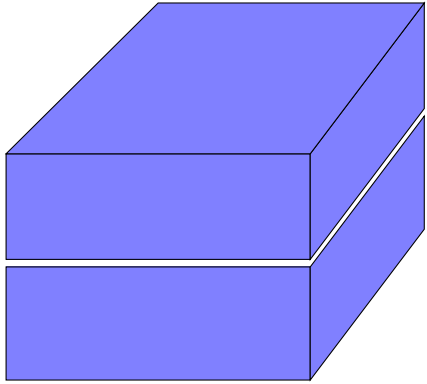
\begin{center}
\begin{tikzpicture}
\filldraw[fill=blue!50] (0,1.6)--(4,1.6)--(4,3.0)--(0,3.0)--cycle;

\filldraw[fill=blue!50] (0,0)--(4,0)--(4,1.5)--(0,1.5)--cycle;

\filldraw[fill=blue!50] (4,1.6)--(5.5,3.6)--(5.5,5.0)--(4,3.0)--cycle;
\filldraw[fill=blue!50] (4,0)--(5.5,2.0)--(5.5,3.5)--(4,1.5)--cycle;

\begin{scope}
\def\rectanglepath{(0,3.0)--(4,3.0)--(5.5,5.0)--(2,5.0)--cycle}
\filldraw[fill=blue!50] \rectanglepath;
\clip\rectanglepath;

\end{scope}

\end{tikzpicture}
\end{center}

On the common boundary 
$\Gamma=\base\cart\{0\}$ (the spacer), $\gamma^+$
is the trace map from above that sends the restriction $\bm{m}_{\vert\base\cart\rbrack0,\hauteursur\lbrack}$
to $\gamma\bm{m}$ on $\Gamma$, and 
$\gamma^-$ is the trace map from below that sends the restriction $\bm{m}_{\vert\base\cart\rbrack-\hauteursous,0\lbrack}$
to $\gamma\bm{m}$ on $\Gamma$. To simplify notations, we consider $\Gamma$
has two sides: $\Gamma^+=\base\cart\{0^+\}$ and $\Gamma^-=\base\cart\{0^-\}$.
By $\Gamma^\pm$, we denote the union of these two sides $\Gamma^+\cup\Gamma^-$. In this
paper, integrating over $\Gamma^\pm$ means integrating over both
sides, while integrating over $\Gamma$ means integrating only once.
On $\Gamma^\pm$, $\gamma$ is the map that sends $\bm{m}$
to its trace on both sides.  The trace map $\gamma^{*}$ is the trace
map that exchange the two sides of $\Gamma$: it maps $\bm{m}$
to $\gamma(\bm{m}\circ s)$ where $s$ 
is the application that sends 
$(x,y,z,t)$ to $(x,y,-z,t)$.

For convenience, we denote by $\bm{\nu}$ the extension to $\Omega$ of
the unitary exterior normal defined
on $\Gamma^\pm$, thus $\bm{\nu}(\bm{x})=-\bm{e}_z$ 
if $z>0$ or if $\bm{x}$ belongs to $\Gamma^+$, 
and $\bm{\nu}(\bm{x})=\bm{e}_z$ if $z<0$ or 
if $\bm{x}$ belongs to $\Gamma^-$.

In this paper, $\mathbb{H}^1(\Omega)$ denotes
$\mathrm{H}^1(\Omega;\mathbb{R}^3)$,
and $\mathbb{L}^2(\Omega)$ denotes
$L^2(\Omega;\mathbb{R}^3)$. 
By $\mathcal{C}^\infty_c(\Omega)$, we denote the set of
$\mathcal{C}^\infty$ functions that have compact support in 
$\Omega$. By $\mathcal{C}^\infty_c(\lbrack0,T\rbrack\cart\Omega)$,
we denote the set of
$\mathcal{C}^\infty$ functions that have compact support in 
$\lbrack0,T\rbrack\cart \Omega$.

\section{The micromagnetic model}\label{sect:MicromagneticModel}

One possible model of ferromagnetism is the micromagnetic model
introduced by W.F~Brown\cite{Brown:microm}. In the micromagnetic model, the
magnetization $\bm{M}$ is the mean at the mesoscopic scale of
the microscopic magnetization and has constant norm $M_s$ in the
ferromagnetic material and is null outside. In this
paper, we only work with the dimensionless magnetization
$\bm{m}=\bm{M}/M_s$.

To each  interaction $p$ present in the ferromagnetic material is
associated an energy $\mathrm{E}_p(\bm{m})$ and an operator
$\mathcal{H}_p$ linked by
\begin{equation*}
\uD\mathrm{E}_p(\bm{m})\cdot\bm{v}=-\int_{\Omega}\mathcal{H}_p(\bm{m}) (\bm{x})\cdot\bm{v}(\bm{x})\ud\bm{x}
\end{equation*}
The vector field $\bm{h}_p=\mathcal{H}_p(\bm{m})$ is the magnetic
effective field associated to interaction $p$. The total energy is the sum
of all the energies associated with every interaction.

These energies completely characterize the stationary problem:
the steady states of the magnetization are the
minimizers of the total energy under the constraint
$\lVert\bm{m}\rVert=1$.

To have an evolution problem, a phenomenological partial differential
equation was introduced in Landau-Lifshitz~\cite{L.L.:Phys}, the Landau-Lifshitz
equation:
\begin{equation*}
\frac{\partial\bm{m}}{\partial t}=-\bm{m}\vect\bm{h}_{\mathrm{tot}}
-\alpha\bm{m}\vect(\bm{m}\vect\bm{h}_{\mathrm{tot}}).
\end{equation*}
where $\bm{h}_{\mathrm{tot}}$ contains all the contributions to the
magnetic effective field. These contributions can either be volumic or
surfacic in nature.

\subsection{Volume energies}
\subsubsection{Exchange}
Exchange is essential in the micromagnetic theory. Without exchange,
there would be no ferromagnetic materials. This
interaction aligns the magnetization over short distances.
In the isotrope and homogenous case,
the exchange energy may be modeled by the following energy
\begin{equation*}
\mathrm{E}_{e}(\bm{m})=\frac{A}{2}\int_\Omega\lVert\nabla\bm{m}\rVert^2\ud\bm{x}.
\end{equation*}
The associated exchange operator is $\mathcal{H}_e(\bm{m})=-A\Lapl\bm{m}$.

\subsubsection{Anisotropy}
Many ferromagnetic materials have a crystalline structure. This
crystalline structure can penalize some directions of magnetization
and favor others. Anisotropy can be modeled by 
\begin{equation*}
\mathrm{E}_{a}(\bm{m})=\frac{1}{2}\int_\Omega(\mathbf{K}(\bm{x})\bm{m}(\bm{x}))\cdot\bm{m}(\bm{x})\ud\bm{x}.
\end{equation*}
where $\mathbf{K}$ is a positive symmetric matrix field.
The associated anisotropy operator is $\mathcal{H}_a(\bm{m})=-\mathbf{K}\bm{m}$.

\subsubsection{Maxwell}
This is the magnetic interaction that comes from Maxwell
equations. The constitutive relations in the ferromagnetic medium are given by:
\begin{equation*}\left\{
\begin{array}{l}
B=\mu_0(\bm{h}+\overline{\bm{m}}),\\
D=\varepsilon_0 \bm{e},
\end{array}
\right.
\end{equation*}
where $\overline{\bm{m}}$ is the extension of $\bm{m}$ by zero outside $\Omega$.

Starting from the Maxwell equations, the magnetic excitation $\bm{h}$ and the electric field
$\bm{e}$ are solutions to the following system:
\begin{align*}
\mu_0\frac{\partial(\bm{h}+\overline{\bm{m}})}{\partial t}+\Rot\bm{e}&=0,\\
\mu_0\frac{\partial\bm{e}}{\partial
  t}+\sigma(\bm{e}+\bm{f})\mathds{1}_\Omega-\Rot\bm{h}&=0.
\end{align*}
As these are evolution equations, initial conditions are needed to
complete the system. 
The energy associated with the Maxwell
interaction is 
\begin{equation*}
E_{\textrm{maxw}}(\bm{h},\bm{e})=\frac{1}{2}\lVert\bm{h}\rVert^2_{\mathrm{L}^2(\mathbb{R}^3)}
+\frac{\varepsilon_0}{2\mu_0}\lVert\bm{e}\rVert^2_{\mathrm{L}^2(\mathbb{R}^3)}.
\end{equation*}

We recall the Law of Faraday:
$\div B=0.$
Here, the constitutive relation reads $B=\mu_0(\bm{h}+\overline{\bm{m}})$. Therefore, in order to satisfy the law of Faraday, we must assume that it is satisfied at initial time. For positive times, by taking the divergence of the first Maxwell's equation, we remark that the divergence free condition is propagated by the system.

\subsection{Surface energies}
When a spacer is present inside a ferromagnetic material, new physical
phenomena may appear in the spacer. These phenomena are modeled 
by surface energies, see M.~Labrune and J.~Miltat~\cite{Labrune.Miltat:wall.structure}.

\subsubsection{Super-exchange}
This surface energy penalizes the jump of the magnetization across the
spacer. It is modeled by a quadratic and a biquadratic term:
\begin{equation}\label{eq:SuperExchangeEnergy}
\mathrm{E}_{se}(\bm{m})=
\frac{J_1}{2}\int_\Gamma\lVert\gamma^+\bm{m}-\gamma^-\bm{m}\rVert^2\ud S(\hat{\bm{x}})
+J_2\int_\Gamma\lVert\gamma^+\bm{m}\vect\gamma^-\bm{m}\rVert^2\ud S(\hat{\bm{x}}).
\end{equation}
The magnetic excitation associated with super-exchange is:
\begin{equation*}
\mathcal{H}_{se}(\bm{m})=\Big(J_1(\gamma^{*}\bm{m}-\gamma\bm{m})
+2J_2\big((\gamma\bm{m}\cdot\gamma^{*}\bm{m})\gamma^{*}\bm{m}
-\lVert\gamma^*\bm{m}\rVert^2\gamma\bm{m}\big)\Big)
\ud S(\Gamma^+\cup\Gamma^-),
\end{equation*}
where $\gamma^*$ is defined in~\S3.
Integration over $\ud S(\Gamma^+\cup\Gamma^-)$ should be understood as
integrating over both faces of the surface $\Gamma$.

\subsubsection{Surface anisotropy}
Surface anisotropy penalizes magnetization that is orthogonal on the
boundary. In the micromagnetic model, it is modeled by a surface energy:
\begin{equation}\label{eq:SurfaceAnisotropyEnergy}
\begin{split}
\mathrm{E}_{sa}(\bm{m})&=\frac{K_s}{2}\int_{\Gamma^+}\lVert\gamma\bm{m}\vect\bm{\nu}\rVert^2\ud S(\hat{\bm{x}})
+\frac{K_s}{2}\int_{\Gamma^-}\lVert\gamma\bm{m}\vect\bm{\nu}\rVert^2\ud S(\hat{\bm{x}})\\
&=\frac{K_s}{2}\int_{\Gamma^\pm}\lVert\gamma\bm{m}\vect\bm{\nu}\rVert^2\ud S(\hat{\bm{x}}).
\end{split}
\end{equation}
The magnetic excitation associated with surface anisotropy is:
\begin{equation*}
\mathcal{H}_{sa}(\bm{m})=K_s\big((\gamma\bm{m}\cdot\bm{\nu})\bm{\nu}-\gamma\bm{m}\big)
\ud S(\Gamma^+\cup\Gamma^-).
\end{equation*}

\subsubsection{New boundary conditions}
Without surface energies, the standard boundary condition is the
homogenous Neumann condition. When surface energies are present, 
the boundary conditions are the ones arising from the stationarity conditions
on the total magnetic energy:
\begin{equation*}
A\gamma\bm{m}\vect\frac{\partial\bm{m}}{\partial\bm{\nu}}=
K_s(\bm{\nu}\cdot\gamma\bm{m})\gamma\bm{m}\vect\bm{\nu}
+J_1\gamma\bm{m}\vect\gamma^{*}\bm{m}
+2J_2(\gamma\bm{m}\cdot\gamma^{*}\bm{m})\gamma\bm{m}\vect\gamma^{*}\bm{m}
\end{equation*}
on the interface $\Gamma^\pm$. A more convincing justification for
these boundary conditions is that they are the ones needed to recover
formally the energy inequality. These boundary conditions are
nonlinear.

\section{The Landau-Lifshitz system}\label{sect:LandauLifshitzSystem}
We consider the following Landau-Lifshitz-Maxwell system:
\begin{subequations}
\begin{align}
\frac{\partial\bm{m}}{\partial t}&=
-\bm{m}\vect\bm{h}_{\mathrm{tot}}^{\mathrm{vol}}-\alpha\bm{m}\vect(\bm{m}\vect\bm{h}_{\mathrm{tot}}^{\mathrm{vol}})\mbox{ in }\mathbb{R}^+\times \Omega,\label{41a}\\
\bm{m}(0,\cdot)&=\bm{m}_0\mbox{ on } \Omega,\\
\lVert\bm{m}\rVert&=1\mbox{ in }\mathbb{R}^+\times \Omega,\\
\frac{\partial\bm{m}}{\partial\bm{\nu}}&= 
0\qquad\text{on $\partial\Omega\setminus\Gamma^\pm$},\\
\begin{split}
\frac{\partial\bm{m}}{\partial\bm{\nu}}&=
\frac{Ks}{A}(\bm{\nu}\cdot\gamma\bm{m})
		(\bm{\nu}-(\bm{\nu}\cdot\gamma\bm{m})\gamma\bm{m})
\\&\phantom{=}
	+\frac{J_1}{A}(\gamma^{*}\bm{m}-
		(\gamma\bm{m}\cdot\gamma^{*}\bm{m})
		\gamma\bm{m})				
\\&\phantom{=}
 	+2\frac{J_2}{A}(\gamma\bm{m}\cdot\gamma^{*}\bm{m})
		(\gamma^{*}\bm{m}-
		(\gamma\bm{m}\cdot\gamma^{*}\bm{m})
		\gamma\bm{m})
		 \qquad\text{on $\mathbb{R}\times\Gamma^\pm$},
\end{split}
\end{align}
\end{subequations}
where $\bm{h}_{\mathrm{tot}}^{\mathrm{vol}}=\bm{h}-\mathbf{K}\bm{m}+A\Lapl\bm{m}$
and $(\bm{e},\bm{h})$ is solution to Maxwell equations:
\begin{subequations}
\begin{align}
\mu_0\frac{\partial(\overline{\bm{m}}+\bm{h})}{\partial t}+\Rot\bm{e}&=0\mbox{ in }\mathbb{R}^+\times \mathbb{R}^3,\\
\varepsilon_0\frac{\partial\bm{e}}{\partial t}+\sigma(\bm{e}+\bm{f})\mathds{1}_\Omega-\Rot\bm{h}&=0\mbox{ in }\mathbb{R}^+\times \mathbb{R}^3,\\
\bm{e}(0,\cdot)&=\bm{e}_0\mbox{ in }\mathbb{R}^3,\\
\bm{h}(0,\cdot)&=\bm{h}_0\mbox{ in } \mathbb{R}^3.
\end{align}
\end{subequations}

We first begin by defining the concept of weak solution to the
Landau-Lifshitz-Maxwell system with surface energies. This 
concept of weak solutions is present 
in~\cite{Alouges.Soyeur:OnGlobal, Carbou.Fabrie:Time, labbe:mumag,
  Santugini:SolutionsLL}. The key point is that the Landau-Lifschitz equation (\ref{41a}) is formally equivalent to the following Landau-Lifschitz-Gilberg equation: 
  
  $$\frac{\partial\bm{m}}{\partial t}-\alpha \bm{m} \vect \frac{\partial\bm{m}}{\partial t}=-(1+\alpha^2)\bm{m}\vect\bm{h}_{\mathrm{tot}}^{\mathrm{vol}},$$
  which is more convenient to obtain the weak formulation defined by:

\begin{definition}[Weak solutions to Landau-Lifshitz-Maxwell with
  surface energies]\label{defin:LLMaxwellWeak}
\begin{subequations} 
Functions $\bm{m}$ in $\mathrm{L}^\infty(0,+\infty;\mathbb{H}^1(\Omega))$
and in $\mathrm{H}^1_{loc}([0,+\infty\lbrack;\mathbb{L}^2(\Omega))$ with
$\frac{\partial\bm{m}}{\partial t}$ in $ \mathbb{L}^2(\mathbb{R}^+\cart\Omega)$,
$\bm{e}$ in 
$\mathrm{L}^\infty(\mathbb{R}^+;\mathbb{L}^2(\mathbb{R}^3))$,
and $\bm{h}$ 
in $\mathrm{L}^\infty(\mathbb{R}^+;\mathbb{L}^2(\mathbb{R}^3))$ 
are said to be weak solutions to the Landau-Lifshitz Maxwell 
system with surface energies if
\begin{enumerate}
\item 
$\lVert \bm{m}\rVert=1$ almost 
everywhere in $\rbrack0,T\lbrack\cart\Omega$.
\item For all $T>0$ and $\bm{\phi}$ in $\mathbb{H}^1(\rbrack0,T\lbrack\cart\Omega)$, 
\begin{equation}\label{eq:WeakFormulationMagnetization}
\begin{split}
&\phantom{=}
\iint_{Q_T}
\frac{\partial\bm{m}}{\partial t}\cdot\bm{\phi}\ud\bm{x}\ud t
-\alpha\iint_{Q_T}
	\left(\bm{m}(t,\bm{x})\vect\frac{\partial\bm{m}}{\partial t}(t,\bm{x})\!\right)
	\cdot\bm{\phi}(t,\bm{x})\ud\bm{x}\ud t 
\\&=
(1+\alpha^2) A\iint_{Q_T}\sum_{i=1}^3
	\left(\bm{m}(t,\bm{x})\vect\frac{\partial \bm{m}}{\partial x_i}(t,\bm{x})\right)
	\cdot\frac{\partial\bm{\phi}}{\partial x_i}(t,\bm{x})\ud\bm{x}\ud t
 \\&\phantom{=}
+(1+\alpha^2)\iint_{Q_T}
	\left(\bm{m}(t,\bm{x})\vect\mathbf{K}(\bm{x})\bm{m}(t,\bm{x})\right)		
	\cdot\bm{\phi}(t,\bm{x})\ud\bm{x}\ud t
\\&\phantom{=}
-(1+\alpha^2)\iint_{Q_T}
	\left(\bm{m}(t,\bm{x})\vect\bm{h}(t,\bm{x})\right)		
	\cdot\bm{\phi}(t,\bm{x})\ud\bm{x}\ud t					
\\&\phantom{=}
-(1+\alpha^2)K_s\iint_{\rbrack0,T\lbrack \cart\Gamma^\pm}
	(\bm{\nu}\cdot\gamma\bm{m})(\gamma\bm{m}\vect\bm{\nu})
	\cdot\gamma\bm{\phi}\ud S(\hat{\bm{x}})\ud t					
\\&\phantom{=}
-(1+\alpha^2)J_1\iint_{\rbrack0,T\lbrack\cart\Gamma^\pm}
	(\gamma\bm{m}\vect\gamma^{*}\bm{m})\cdot\gamma\bm{\phi}\ud S(\hat{\bm{x}})\ud t		
\\&\phantom{=}
	-2(1+\alpha^2)J_2
		\iint_{\rbrack0,T\lbrack\cart\Gamma^\pm}
	(\gamma\bm{m}\cdot\gamma^{*}\bm{m})(\gamma\bm{m}\vect\gamma^{*}\bm{m})
	\cdot\gamma\bm{\phi}\ud S(\hat{\bm{x}})\ud t.
\end{split}
\end{equation}
\item In the sense of traces,
$\bm{m}(0,\cdot)=\bm{m}_0$.
\item For all $\bm{\psi}$ in $\mathcal{C}^\infty_c(\lbrack0,+\infty\lbrack,\mathbb{R}^3)$:
\begin{multline}\label{eq:WeakFormulationExcitation}
-\mu_0\iint_{\mathbb{R}^+\cart\mathbb{R}^3}(\bm{h}+\bm{m})\cdot
\frac{\partial\bm{\psi}}{\partial t}\ud\bm{x}\ud t
+\iint_{\mathbb{R}^+\cart\mathbb{R}^3}\bm{e}\cdot\Rot\bm{\psi}\ud\bm{x}\ud t
=\\=
\mu_0\int_{\mathbb{R}^3}(\bm{h}_0+\bm{m}_0)\cdot\bm{\psi}_0 \ud\bm{x}
\end{multline}
\item For all $\bm{\Theta}$ in $\mathcal{C}^\infty_c(\lbrack0,+\infty\lbrack\cart\mathbb{R}^3)$:
\begin{multline}\label{eq:WeakFormulationElectric}
-\varepsilon_0\iint_{\mathbb{R}^+\cart\mathbb{R}^3}\bm{e}\cdot
\frac{\partial\bm{\Theta}}{\partial t}\ud\bm{x}\ud t
-\iint_{\mathbb{R}^+\cart\mathbb{R}^3}\bm{h}\cdot\Rot\bm{\Theta}\ud\bm{x}\ud t
+\sigma\iint_{\mathbb{R}^+\cart\Omega}(\bm{e}+\bm{f})
\cdot\bm{\Theta}
\ud\bm{x}\ud t
=\\=
\varepsilon_0\int_{\mathbb{R}^3}\bm{e}_0\cdot\bm{\Theta}_0 \ud\bm{x}.
\end{multline}
\item The following energy inequality holds
\begin{equation}\label{eq:EnergyInequality}
\begin{split}
\mathrm{E}(\bm{m}(T),\bm{h}(T),\bm{e}(T))+\frac{\alpha}{1+\alpha^2}\iint_{Q_T}
\left\vert \frac{\partial \bm{m}}{\partial t}\right\vert^2 \ud\bm{x}
\ud t 
&\\
+\frac{\sigma}{\mu_0}\int_0^T\lVert\bm{e}\rVert^2_{\mathbb{L}^2(\Omega)}\ud t+ 
\frac{\sigma}{\mu_0} \iint_{Q_T} \bm{e}\cdot\bm{f}\ud\bm{x} \ud t 
&\leq \mathrm{E}(\bm{m}_0,\bm{h}_0,\bm{e}_0),
\end{split}
\end{equation}
where 
\begin{equation*}
\begin{split}
\mathrm{E}(\bm{m},\bm{h},\bm{e})&=\frac{A}{2}\int_\Omega\lVert\nabla\bm{m}\rVert^2\ud\bm{x}
+\frac{1}{2}\int_\Omega(\mathbf{K}(\bm{x})\bm{m}(\bm{x}))\cdot\bm{m}(\bm{x})\ud\bm{x}\\
&\phantom{=}+\frac{\varepsilon_0}{2\mu_0}\int_{\mathbb{R}^3}\lVert\bm{e}(\bm{x})\rVert^2
+\frac{1}{2}\int_{\mathbb{R}^3}\lVert\bm{h}(\bm{x})\rVert^2
+\frac{K_s}{2}\int_{\Gamma^+\cup\Gamma^-}\lVert\gamma^+\bm{m}\vect\bm{\nu}\rVert^2\ud S(\bm{x})
\\
&\phantom{=}
+\frac{J_1}{2}\int_\Gamma\lVert\gamma^+\bm{m}-\gamma^-\bm{m}\rVert^2\ud\bm{x}
+J_2\int_\Gamma\lVert\gamma^+\bm{m}\vect\gamma^-\bm{m}\rVert^2\ud\bm{x}.
\end{split}
\end{equation*}
\end{enumerate}
\end{subequations}
\end{definition}

Our first result states the existence of a global in time weak solution to the Laudau-Lifschitz-Maxwell system .

\begin{theorem}\label{theo:ExistenceWeakLandauLifshitzMaxwellSurfaceEnergies}
Let $\bm{m}_0$ be in $\mathbb{H}^1(\Omega)$ such that
$\lVert\bm{m}_0\rVert=1$ almost everywhere in $\Omega$. Let
$\bm{h}_0$ and $\bm{e}_0$ be in $\mathbb{L}^2(\Omega)$. Let $\bm{f}$
be in $\mathbb{L}^2(\mathbb{R}^+\cart\Omega)$
Suppose $\Div(\bm{h}_0+\overline{\bm{m}_0})=0$ in $\mathbb{R}^3$, where
$\overline{\bm{m}_0}$ is the extension of $\bm{m}_0$ by $0$ outside $\Omega$.
Then, there exists at least one weak solution to
the Landau-Lifshitz-Maxwell system in the sense of
Definition~\ref{defin:LLMaxwellWeak}. 
\end{theorem}
Uniqueness is unlikely as the solution isn't unique when only the
exchange energy is present, see~\cite{Alouges.Soyeur:OnGlobal}.

\vspace{3mm}
In our second result we characterize the $\omega$-limit set of a trajectory. The definition is the following:

\begin{definition}
\label{defomega}
Let $(\bm{m},\bm{h},\bm{e})$ be a weak solution of the Landau-Lifschitz-Maxwell system given by Theorem ~\ref{theo:ExistenceWeakLandauLifshitzMaxwellSurfaceEnergies}. We call $\omega$-limit set of this trajectory the set:
\begin{equation*}\omega(\bm{m})=\left\{ v\in H^1(\Omega), \exists (t_n)_n, \; \lim_{n\rightarrow +\infty} t_n=+\infty, \; \bm{m}(t_n,.)\rightharpoonup v \mbox{ weakly  in } H^1(\Omega)\right\}.\end{equation*}
\end{definition}

We remark that $m\in L^\infty(0,+\infty;\mathbb{H}^1(\Omega))$ so that $\omega(m) $ is non empty. 

\vspace{2mm}
\begin{theorem}
\label{theo-omegalim}
Let $(\bm{m},\bm{e},\bm{h})$ be a weak solution of the Landau-Lifschitz-Maxwell system given by Theorem ~\ref{theo:ExistenceWeakLandauLifshitzMaxwellSurfaceEnergies}. Let $\bm u\in \omega(\bm m)$. Then $\bm u$ satisfies:
\begin{enumerate}
\item $\bm{u}\in \mathbb{H}^1(\Omega)$, $\vert \bm{u}\vert =1$ almost everywhere,
\item for all $\bm{\varphi}\in \mathbb{H}^1(\Omega)$,
\begin{equation}\label{eq:WeakFormulationMagnetizationOmega}
\begin{split}
0&=
 A\int_{\Omega}\sum_{i=1}^3
	\left(\bm{u}(\bm{x})\vect\frac{\partial \bm{u}}{\partial x_i}(\bm{x})\right)
	\cdot\frac{\partial\bm{\varphi}}{\partial x_i}(t,\bm{x})\ud\bm{x}
 \\&\phantom{=}
+\int_{\Omega}\
	\left(\bm{u}(\bm{x})\vect\mathbf{K}(\bm{x})\bm{u}(\bm{x})\right)		
	\cdot\bm{\varphi}(\bm{x})\ud\bm{x}
\\&\phantom{=}
-\int_{\Omega}\
	\left(\bm{u}(\bm{x})\vect\bm{H}(\bm{x})\right)		
	\cdot\bm{\varphi}(\bm{x})\ud\bm{x}				
\\&\phantom{=}
-K_s\int_{(\Gamma^\pm)}
	(\bm{\nu}\cdot\gamma\bm{u})(\gamma\bm{u}\vect\bm{\nu})
	\cdot\gamma\bm{\varphi}\ud S(\hat{\bm{x}})				
\\&\phantom{=}
-J_1\int_{(\Gamma^\pm)}
	(\gamma\bm{u}\vect\gamma^{*}\bm{m})\cdot\gamma\bm{\varphi}\ud S(\hat{\bm{x}})		
\\&\phantom{=}
	-2J_2
		\int_{\Gamma^\pm}
	(\gamma\bm{u}\cdot\gamma^{*}\bm{u})(\gamma\bm{u}\vect\gamma^{*}\bm{u})
	\cdot\gamma\bm{\varphi}\ud S(\hat{\bm{x}}).
\end{split}
\end{equation}
\item $\bm H$ is deduced from $\bm u$ by the relations:
\begin{equation*}\div (\bm H +\overline{ \bm u})=0 \mbox{ and }\curl \bm H =0 \mbox{ in }\D'(\R^3).\end{equation*}

\end{enumerate}
\end{theorem}

\section{Technical prerequisite results on Sobolev Spaces}\label{sect:prerequisites}
In this section, we remind the reader about some useful previously
known results on Sobolev Spaces that we use in this paper. In the
whole section $\mathcal{O}$ is any bounded open set of $\mathbb{R}^3$,
regular enough for the usual embeddings result to hold. For example,
it is enough that $\mathcal{O}$ 
satisfy the cone property, see\cite[\S4.3]{Adams:SobolevSpaces}.

We start with Aubin's lemma \cite{Aubin:TheoremeCompacite}, as extended in~\cite[Corollary 4]{Simon:CompactSets}.
\begin{lemma}[Aubin's lemma]\label{lemma:AubinSimon}
Let $X\subset\subset B\subset Y$ be Banach spaces. Let $F$ be bounded
in $L^p(0,T;X)$. Suppose $\{\partial_tu, u\in F\}$ is bounded in
$L^r(0,T;Y)$. Suppose for all $t$ in $ $.
\begin{itemize} 
\item If $r\geq 1$ and $1\leq p<+\infty$, then $F$ is a compact subset of $L^p(0,T;X)$ .
\item If $r>1$ and $p=+\infty$, then $F$ is a compact subset of $\mathcal{C}(0,T;B)$.
\end{itemize}
\end{lemma}

\begin{lemma}\label{lemma:compactCL2imbedding}
For all $T>0$, the imbedding from $H^1(\rbrack0,T\lbrack\cart\mathcal{O})$ to 
$\mathcal{C}(\lbrack0,T\rbrack,L^2(\mathcal{O}))$ is compact.
\end{lemma}
\begin{proof}
Use the Aubin's lemma, see~\cite[Corollary 4]{Simon:CompactSets}, 
extended to the case $p=+\infty$, with
$X=H^1(\mathcal{O})$ and  $B=Y=L^2(\Omega)$.
\end{proof}

\begin{lemma}
Let $u$ belong to $H^1(\rbrack0,T\lbrack\cart\mathcal{O})\cap
L^\infty(\rbrack0,T\lbrack;{H}^1(\mathcal{O}))$,
then $u$ belongs to
$\mathcal{C}(\lbrack0,T\rbrack;\mathbb{H}^1_\omega(\mathcal{O}))$
where $H_\omega^1(\mathcal{O})$ is the space $H^1(\mathcal{O})$ but with the weak topology.
\end{lemma}
\begin{proof}
The function $u$, belongs to
$\mathcal{C}(\lbrack0,T\rbrack,\mathbb{L}^2(\mathcal{O}))$. Let now $(t_n)_n$ be a
sequence in $\lbrack0,T\rbrack$ converging to $t$. Then, 
$u(t_n,\cdot)$ converges to
$u(t,\cdot)$ in $L^2(\mathcal{O})$. Also, the sequence
$(u(t_n,\cdot))_{n\in\mathbb{N}}$ is bounded in
${H}^1(\mathcal{O})$, therefore from any subsequence of 
$(u(t_n,\cdot))_{n\in\mathbb{N}}$, one can extract a subsequence
that converges weakly in $H^1(\mathcal{O})$. The only possible limit is
$u(t,\cdot)$  therefore the whole sequence converges weakly in $H^1(\mathcal{O})$.
\end{proof}

\begin{lemma}\label{lemma:SameSubsequenceGivenTime}
Let $(u_n)_{n\in\mathbb{N}}$ be bounded in
$H^1(\rbrack0,T\lbrack\cart\mathcal{O})$ and in 
$L^\infty(\rbrack0,T\lbrack;H^1(\mathcal{O}))$. Let
$(u_{n_k})_{k\in\mathbb{N}}$
be a subsequence which converges weakly to
some $u$ in $H^1(\rbrack0,T\lbrack\cart\mathcal{O})$. Then, for all $t$ in
$\lbrack0,T\rbrack$, the same subsequence $u_{n_k}(t,\cdot)$ converges weakly to
$u(t,\cdot)$ in $H^1(\mathcal{O})$.
\end{lemma}
\begin{proof}
For all $t$ in $\lbrack0,T\rbrack$, $u_{n_k}(t,\cdot)$ converges strongly
to $u(t,\cdot)$ in $L^2(\mathcal{O})$. Therefore, any subsequence 
$u_{n_{k_j}}(t,\cdot)$ that converges weakly in $H^1(\mathcal{O})$ has 
$u(t,\cdot)$ for limit. Since $u_{n_k}(t,\cdot)$ is bounded
in $H^1(\mathcal{O})$, from any subsequence of $u_{n_k}(t,\cdot)$, one
can extract a further subsequence  that converges weakly in
$H^1(\mathcal{O})$, therefore, for all $t$ in $\lbrack0,T\rbrack$,
the whole subsequence $u_{n_k}(t,\cdot)$
converges weakly to $u(t,\cdot)$ in $H^1(\mathcal{O})$.
\end{proof}

\section{Proof of
  Theorem~\ref{theo:ExistenceWeakLandauLifshitzMaxwellSurfaceEnergies}}
\label{sect:proofExistenceWeakLLMSurfaceEnergies}
\subsection{Idea of the proof}
We proceed as in \cite{Carbou.Fabrie:Time}
and~\cite{Santugini:SolutionsLL} and combine the ideas of both
papers. 
We start by extending the surface energies to
a thin layer of thickness $2\eta>0$. 

As in~\cite{Santugini:SolutionsLL}, we consider the operator
\begin{equation}
\begin{gathered}
\mathcal{H}_s^\eta:\mathbb{H}^1(\Omega) \cap \mathbb{L}^\infty(\Omega)
\to\mathbb{H}^1(\Omega) \cap \mathbb{L}^\infty(\Omega)\\
\bm{m}\mapsto\frac{1}{2\eta}\begin{cases}
	0&\text{in $\mathbb{R}^3\setminus(\,\base\cart(\Interv{}\setminus\Interveta{\eta}{})\;)$},	\\
	\begin{gathered}
	2K_s((\bm{m}\cdot\bm{\nu})\bm{\nu}-\bm{m})+2J_1(\bm{m}^{*}-\bm{m})\\
	+4J_2\big((\bm{m}\cdot\bm{m}^{*})\bm{m}^{*}
		-\lVert\bm{m}^{*}\rVert^2\bm{m}\big)
	\end{gathered}
	&\text{in $\base\cart(\Interv{}
		\setminus\Interveta{\eta}{})$},
	\end{cases}
\end{gathered}
\end{equation} 
where $\bm{m}^*$ is the reflection of $\bm{m}$, \textit{i.e.}
$\bm{m}^*(x,y,z,t)=\bm{m}(x,y,-z,t)$, see Figure~\ref{fig:boundarylayer}.
\begin{figure}[ht]
\begin{center}
\begin{tikzpicture}
\draw (0,-2) rectangle (5,2);
\fill[lightgray] (0,0) rectangle (5, 0.2) ;
\fill[lightgray] (0,0) rectangle (5,-0.2) ;
\draw (0,0) -- (5,0);
\draw[<->] (5.2,-0.2) -- (5.2, 0.2);
\node at (5.4,0) {$\eta$};
\end{tikzpicture}
\caption{Artificial boundary layer}\label{fig:boundarylayer}
\end{center}
\end{figure}
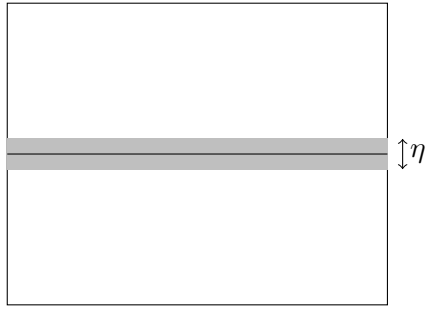
The associated energy is:
\begin{equation}
\begin{split}
\mathrm{E}_s^\eta(\bm{m})&=\frac{K_s}{2\eta}
\int_{\base\cart(\Interv{}\setminus\Interveta{\eta}{})}
\left(\lVert\bm{m}\rVert^2-(\bm{m}\cdot\bm{\nu})^2\right)\ud\bm{x}\\
&\phantom{=}+\frac{J_1}{2\eta}
\int_{\base\cart(\Interv{}\setminus\Interveta{\eta}{})}
\left(\frac{\lVert \bm{m}\rVert^2+\lVert \bm{m}^{*}\rVert^2}{2}-(\bm{m}\cdot\bm{m}^{*})\right)	\ud\bm{x}\\
&\phantom{=}+\frac{J_2}{2\eta}				
\int_{\base\cart\Interv{}\setminus\Interveta{\eta}{}}
\left(\lVert\bm{m}^{*}\rVert^2\lVert\bm{m}\rVert^2-(\bm{m}\cdot\bm{m}^{*})^2\right)
	\ud\bm{x}.
\end{split}
\end{equation}
This energy will replace the surfacic ones~\eqref{eq:SuperExchangeEnergy} and~\eqref{eq:SurfaceAnisotropyEnergy}. The idea
is to consider the Landau-Lifshitz-Maxwell system with homogenous
Neumann boundary conditions with the excitation containing this new
component then have $\eta$ tend to $0$.

We consider the doubly penalized problem:
\begin{subequations}\label{subeq:PenalizedLL}
\begin{align}
\begin{split}
\alpha\frac{\partial\bm{m}_{k,\eta}}{\partial t}+\bm{m}_{k,\eta}\vect\frac{\partial\bm{m}_{k,\eta}}{\partial t}&=
(1+\alpha^2)(A\Lapl\bm{m}-\mathbf{K}\bm{m}+\bm{h}_{k,\eta}+\mathcal{H}_s^\eta(\bm{m}_{k,\eta}))
\\&\phantom{=}
-k(1+\alpha^2)((\lVert\bm{m}_{k,\eta}\rVert^2-1)\bm{m}_{k,\eta}),
\end{split}\\
\frac{\partial\bm{m}_{k,\eta}}{\partial\bm{\nu}}&=0 \quad\text{on $\partial\Omega$},\\
\bm{m}_{k,\eta} (0,\cdot)&=\bm{m}_0,
\end{align}
\end{subequations}
with Maxwell equations:
\begin{subequations}
\begin{align}
\varepsilon_0\frac{\partial\bm{e}_{k,\eta}}{\partial t}+\sigma(\bm{e}_{k,\eta}+\bm{f})\mathds{1}_\Omega-\Rot\bm{h}_{k,\eta}&=0,\\
\mu_0\frac{\partial(\bm{m}_{k,\eta}+\bm{h}_{k,\eta})}{\partial t}+\Rot\bm{e}_{k,\eta}&=0,\\
\bm{e}_{k,\eta} (0,\cdot)&=\bm{e}_0,\\
\bm{h}_{k,\eta} (0,\cdot)&=\bm{h}_0.
 \end{align}
\end{subequations}

The basic idea is to prove the existence of weak solutions to the penalized problem via 
Galerkin, then have $k$ tend to $+\infty$ to satisfy the local norm
constraint on the magnetization, then have $\eta$ tend to $0$ to transform the homogenous Neumann
boundary condition into the nonlinear condition above. 

\subsection{First Step of Galerkin's method }
As
in~\cite{Alouges.Soyeur:OnGlobal} we consider the eigenvectors
$(v_j)_{j\geq1}$ of the Laplace operator
with Neumann homogenous conditions. This basis is, up to a
renormalisation, an hilbertian basis for the spaces
$\mathbb{L}^2(\Omega)$, $\mathbb{H}^1(\Omega)$, and
 $\{\bm{u}\in
\mathbb{H}^2(\Omega),\frac{\partial\bm{u}}{\partial\bm{\nu}}=0\}$. The
 eigenvectors $v_k$ all belong to $\mathcal{C}^\infty(\overline{\Omega};\mathbb{R}^3)$.
We call $V_n$ the space spanned by $(v_j)_{1\leq j\leq n}$.
As in \cite{Carbou.Fabrie:Time}, we
consider an hilbertian basis $(\bm{\omega}_j)_{j\geq1}$ of
$\mathrm{L}^2(\mathbb{R}^3;\mathbb{R}^3)$ such that every $\bm{\omega}_j$ belongs
to $\mathcal{C}^\infty_c(\mathbb{R}^3;\mathbb{R}^3)$.
We call $W_n$ the space spanned by $(\bm{\omega}_j)_{0\leq j\leq n}$.

 Set $n\geq1$, $\eta>0$ and $k>0$.
We search for 
$\bm{m}_{n,k,\eta}$ in $H^1(\mathbb{R}^+;(V_n)^3)$,
$\bm{h}_{n,k,\eta}$ in $H^1(\mathbb{R}^+;W_n)$, and
$\bm{e}_{n,k,\eta}$ in $H^1(\mathbb{R}^+;W_n)$
such that
\begin{subequations}\label{subeq:EqDiffmhenketa}
\begin{equation}
 \begin{split}
\alpha\frac{\ud\bm{m}_{n,k,\eta}}{\ud t}&=-\mathcal{P}_{V_n}(\bm{m}_{n,k,\eta}\vect \frac{\ud\bm{m}_{n,k,\eta}}{\ud t})
\\&\phantom{=} +(1+\alpha^2)\mathcal{P}_{V_n}(A\Lapl\bm{m}_{n,k,\eta}-\mathbf{K}\bm{m}_{n,k,\eta})
\\&\phantom{=}
 +(1+\alpha^2)\mathcal{P}_{V_n}(\bm{h}_{n,k,\eta}
+\mathcal{H}_s^\eta(\bm{m}_{n,k,\eta}))
\\&\phantom{=}
 -(1+\alpha^2)k\mathcal{P}_{V_n}((\lVert\bm{m}_{n,k,\eta}\rVert^2-1)\bm{m}_{n,k,\eta}),
 \end{split}
\end{equation}
and
\begin{equation}
\begin{split}
\mu_0\frac{\ud\bm{h}_{n,k,\eta}}{\ud
  t}&=-\mu_0\mathcal{P}_{W_n}\left(\frac{\ud\bm{m}_{n,k,\eta}}{\ud
    t}\right)
+\mathcal{P}_{W_n}(\Rot\bm{e}_{n,k,\eta}).
\end{split}
\end{equation}
and
\begin{equation}
\begin{split}
\varepsilon_0\frac{\ud\bm{e}_{n,k,\eta}}{\ud t}&=
-\mathcal{P}_{W_n}(\Rot\bm{h}_{n,k,\eta})
-\mathcal{P}_{W_n}(\mathds{1}_\Omega(\bm{e}_{n,k,\eta}+\bm{f})),
\end{split}
\end{equation}
\end{subequations} 
with the inital conditions:
\begin{subequations}\label{subeq:EqDiffCondInitmhenketa}
\begin{align}
\bm{m}_{n,k,\eta}(0,\cdot)=\mathcal{P}_{V_n}(\bm{m}_0),\\
\bm{h}_{n,k,\eta}(0,\cdot)=\mathcal{P}_{W_n}(\bm{h}_0),\\
\bm{e}_{n,k,\eta}(0,\cdot)=\mathcal{P}_{W_n}(\bm{e}_0),
\end{align}
\end{subequations} 
where $\mathcal{P}_{V_n}$ is the orthogonal projection on $V_n$ in
$\mathrm{L}^2(\Omega)$ and $\mathcal{P}_{W_n}$ is the orthogonal projection on $W_n$ in
$\mathbb{L}^2(\Omega;\mathbb{R}^3))$. Let $\mathbf{a}(t)=(\bm{a}_i(t))_{1\leq
  i\leq n}$,
$\mathbf{b}=(b_i)_{1\leq i\leq n}$ and $\mathbf{c}(t)=(c_i(t))_{1\leq
  i\leq n}$ be the coefficients of $\bm{m}_{n,k,\eta}(t,\cdot)$,
$\bm{h}_{n,k,\eta}(t,\cdot)$ and $\bm{e}_{n,k,\eta}(t,\cdot)$ in the decomposition
\begin{align*}
\bm{m}_{n,k,\eta}(t,\cdot)&=\sum_{i=1}^n\bm{a_i}(t)v_i,\\
\bm{h}_{n,k,\eta}(t,\cdot)&=\sum_{i=1}^nb_i(t)\bm{\omega}_i,\\
\bm{e}_{n,k,\eta}(t,\cdot)&=\sum_{i=1}^nc_i(t)\bm{\omega}_i.
\end{align*}
Then, System~\eqref{subeq:EqDiffmhenketa} is equivalent to
\begin{subequations}\label{subeq:EqDiffmhenketaCoeffABC}
\begin{align}
\frac{\ud\mathbf{a}}{\ud
 t}+\mathbf{\phi}(\mathbf{a},\frac{\ud\mathbf{a}}{\ud t})&=F_{\bm{m}}(\mathbf{a},\mathbf{b}),\\
\frac{\ud(\mathbf{b}+L\mathbf{a})}{\ud  t}&=F_{\bm{h}}(\mathbf{c}),\\
\frac{\ud\mathbf{c}}{\ud
 t}&=F_{\bm{e}}(\bm{h}_{n,k,\eta},\bm{e}_{n,k,\eta})+\mathbf{f}^*,
\end{align}
\end{subequations}
where $L$ is linear, $F_{\bm{m}}$, $F_{\bm{h}}$ and $F_{\bm{e}}$ are polynomial thus
of class $\mathcal{C}^\infty$, and $\bm{f}^*$ is in
$\mathrm{L}^2(\mathbb{R}^+;\mathbb{R}^n)$. 
These are supplemented by initial conditions
\begin{align}\label{eq:EqDiffCondInitmnketaCoeffABC}
\mathbf{a}(0,\cdot)&=\mathbf{a}_0,&
\mathbf{b}(0,\cdot)&=\mathbf{b}_0,&
\mathbf{c}(0,\cdot)&=\mathbf{c}_0,
\end{align}
where $\mathbf{a}_0$, $\mathbf{b}_0$, and $\mathbf{c}_0$ are obtained
by orthogonal projection of $\bm{m}_0$,  $\bm{h}_0$,  $\bm{e}_0$ over the $v_i$
or the $\bm{\omega}_i$.  As $\mathbf{\phi}(\cdot,\cdot)$ is bilinear
continuous and $\mathbf{\phi}(\mathbf{a},\cdot)$ is
antisymmetric, the linear application $\mathrm{Id}-\mathbf{\phi}(\mathbf{a},\cdot)$ is invertible. 
Finally $\bm{f}^{*}$ is $\mathrm{L}^2$.
Therefore, by the Carathéorody theorem,
System~\eqref{subeq:EqDiffmhenketaCoeffABC} has local
solutions. Therefore, there exists  $T^*>0$ and $\bm{m}_{n,k,\eta}$ in
$\mathrm{H}^1(\rbrack0,T^*\lbrack;(V_n)^3)$, $\bm{h}_{n,k,\eta}$ in
$\mathrm{H}^1(\rbrack0,T^*\lbrack;W_n)$ 
and $\bm{e}_{n,k,\eta}$ in $\mathrm{H}^1(\rbrack0,T^*\lbrack;W_n)$  that 
satisfy~\eqref{subeq:EqDiffmhenketa}
and~\eqref{subeq:EqDiffCondInitmhenketa}.

Multiplying~\eqref{subeq:EqDiffmhenketa} by test functions and
integrating by part yields:
\begin{subequations}\label{subeq:WeakFormulationGalerkin}
\begin{equation}\label{eq:WeakFormulationGalerkinMagnetization}
\begin{split}
&\phantom{=}
\alpha\iint_{Q_T}
\frac{\partial\bm{m}_{n,k,\eta}}{\partial t}\cdot\bm{\phi}\ud\bm{x}\ud t
+\iint_{Q_T}
	\left(\bm{m}_{n,k,\eta}\vect\frac{\partial\bm{m}_{n,k,\eta}}{\partial t}\right)
	\cdot\bm{\phi}\ud\bm{x}\ud t 
\\&=
-(1+\alpha^2) A\iint_{Q_T}\sum_{i=1}^3
	\frac{\partial \bm{m}_{n,k,\eta}}{\partial x_i}\cdot\frac{\partial\bm{\phi}}{\partial x_i}\ud\bm{x}\ud t
 \\&\phantom{=}
-(1+\alpha^2)\iint_{Q_T}(
 	\mathbf{K}(\bm{x}) \bm{m}_{n,k,\eta}(\bm{x}))
 	\cdot\bm{\phi}\ud\bm{x}\ud t
\\&\phantom{=}
+(1+\alpha^2)\iint_{Q_T}
	\bm{h}_{n,k,\eta}
	\cdot\bm{\phi}\ud\bm{x}\ud t					
\\&\phantom{=}
-(1+\alpha^2)k\iint_{Q_T}(\lVert\bm{m}_{n,k,\eta}\rVert^2-1)\bm{m}_{n,k,\eta}\cdot\bm{\phi}\ud\bm{x}\ud t
\\&\phantom{=}
        +(1+\alpha^2)\frac{K_s}{\eta}\iint_{\rbrack0,T\lbrack\cart(B\cart\rbrack-\eta,\eta\lbrack)}
        ((\bm{\nu}\cdot\bm{m}_{n,k,\eta})\bm{\nu}-\bm{m}_{n,k,\eta})
	\cdot\bm{\phi}\ud\bm{x}\ud t					
\\&\phantom{=}
+(1+\alpha^2) \frac{ J_1}{\eta}\iint_{\rbrack0,T\lbrack\cart(B\cart\rbrack-\eta,\eta\lbrack)}
	(\bm{m}^{*}_{n,k,\eta}-\bm{m}_{n,k,\eta})\cdot\bm{\phi}\ud\bm{x}\ud t		
\\&\phantom{=}
	+2(1+\alpha^2)\frac{J_2}{\eta}
		\iint_{\rbrack0,T\lbrack\cart(B\cart\rbrack-\eta,\eta\lbrack)}
	\left((\bm{m}_{n,k,\eta}\cdot\bm{m}_{n,k,\eta}^{*})\bm{m}^{*}_{n,k,\eta}
        -\lVert\bm{m}_{n,k,\eta}^*\rVert^2\bm{m}_{n,k,\eta}\right)
	\cdot\bm{\phi}\ud\bm{x}\ud t,
\end{split} 
\end{equation} 
for all $\bm{\phi}$  in $\mathcal{C}^\infty(\lbrack0,T^*\rbrack,V_n^3)$. And
\begin{equation}\label{eq:WeakFormulationGalerkinExcitation}
\begin{split} 
\mathrm{\mu}_0\iint_{\rbrack0,T\lbrack\cart\mathbb{R}^3}\left(\frac{\partial\bm{h}_{n,k,\eta}}{\partial t}+
\frac{\partial\bm{m}_{n,k,\eta}}{\partial t}\right)\cdot \bm{\psi}\ud\bm{x}\ud t
+\iint_{\rbrack0,T\lbrack\cart\mathbb{R}^3}\Rot\bm{e}_{n,k,\eta}\cdot\bm{\psi}\ud\bm{x}\ud t&=0,
\end{split}
\end{equation}
for all $\bm{\psi}$
in $\mathcal{C}^\infty(\lbrack0,T^*\rbrack,W_n)$. And
\begin{equation}\label{eq:WeakFormulationGalerkinElectric}
\begin{split} 
\mathrm{\varepsilon}_0\iint_{\rbrack0,T\lbrack\cart\mathbb{R}^3}\frac{\partial\bm{e}_{n,k,\eta}}{\partial
    t}\cdot\bm{\Theta}\ud\bm{x}\ud t-\iint_{\rbrack0,T\lbrack\cart\mathbb{R}^3}\Rot\bm{h}_{n,k,\eta}\cdot\bm{\Theta}\ud\bm{x}\ud
  t
&\\
+\sigma\iint_{Q_T}(\bm{e}_{n,k,\eta}+\bm{f})\cdot\bm{\Theta}\ud\bm{x}\ud t&=0,
\end{split}
\end{equation}
for all $\bm{\Theta}$ in
 $\mathcal{C}_c^\infty(\lbrack0,T^*\rbrack,W_n)$. 

By density, \eqref{subeq:WeakFormulationGalerkin} also
holds if  $\bm{\phi}$ belongs to $\mathrm{L}^2(\rbrack0,T^*\lbrack;V_n^3)$, $\bm{\psi}$
belongs to $\mathrm{L}^2(\rbrack0,T^*\lbrack,W_n)$, and $\bm{\Theta}$
belongs to $\mathrm{L}^2(\rbrack0,T^*\lbrack,W_n)$.
\end{subequations}
As in~\cite{Carbou.Fabrie:Time}, set
$\bm{\phi}=\frac{\partial\bm{m}_{n,k,\eta}}{\partial t}$ in~\eqref{eq:WeakFormulationGalerkinMagnetization},
we obtain 
\begin{equation*}
\begin{split}
&\phantom{=}\frac{A}{2}\int_{\Omega}\lVert\nabla\bm{m}_{n,k,\eta}(T,\bm{x})\rVert^2\ud\bm{x}
+\frac{1}{2}\int_{\Omega}(\mathbf{K}(\bm{x})\bm{m}_{n,k,\eta}(T,\bm{x}))\cdot\bm{m}(T,\bm{x})\ud\bm{x}
\\&\phantom{=}
+\frac{k}{4}\int_{\Omega}(\lVert\bm{m}_{n,k,\eta}(T,\bm{x}))\rVert^2-1)^2\ud\bm{x}
-\iint_{Q_T}\bm{h}_{n,k,\eta}\cdot\frac{\partial\bm{m}_{n,k,\eta}}{\partial
  t}\ud\bm{x}\ud t
\\&\phantom{=}
+\mathrm{E}_s^\eta(\bm{m}_{n,k,\eta}(T,\cdot))
+\frac{\alpha}{1+\alpha^2}\iint_{Q_T}\left\lVert\frac{\partial\bm{m}_{n,k,\eta}}{\partial
    t}\right\rVert^2 \ud\bm{x}\ud t\\
&\leq
\frac{A}{2}\int_{\Omega}\lVert\nabla\mathcal{P}_n(\bm{m}_0)\rVert^2\ud\bm{x}
+\frac{1}{2}\int_{\Omega}(\mathbf{K}(\bm{x}) \mathcal{P}_{V_n}(\bm{m}_0))\cdot\mathcal{P}_{V_n}(\bm{m}_0)\ud\bm{x}
\\&\phantom{=}
+\frac{k}{4}\int_{\Omega}(\lVert \mathcal{P}_{V_n}(\bm{m}_0))\rVert^2-1)^2\ud\bm{x}
+\mathrm{E}_s^\eta(\mathcal{P}_{V_n}(\bm{m}_0)).
\end{split}
\end{equation*}
Set $\bm{\psi}=\bm{h}_{n,k,\eta}$ in
\eqref{eq:WeakFormulationGalerkinExcitation}, we obtain
\begin{equation*}
\begin{split}
&\phantom{=}
\frac{\mathrm{\mu}_0}{2}\int_{\mathbb{R}^3}\lVert\bm{h}_{n,k,\eta}(T,\bm{x})\rVert^2 \ud\bm{x}\ud t
+\mu_0\iint_{Q_T}\frac{\partial\bm{m}_{n,k,\eta}}{\partial t}\cdot
\bm{h}_{n,k,\eta}\ud\bm{x}\ud t
\\&\phantom{=}
+\iint_{\rbrack0,T\lbrack\cart\mathbb{R}^3}\bm{h}_{n,k,\eta}\cdot\Rot\bm{e}_{n,k,\eta}\ud\bm{x}\ud t
\\&\leq \frac{\mathrm{\mu}_0}{2}\int_{\mathbb{R}^3}\lVert\mathcal{P}_{W_n}(\bm{h}_0)\rVert^2 \ud\bm{x}
,
\end{split}
\end{equation*}
Set $\bm{\Theta}=\bm{e}_{n,k,\eta}$ in \eqref{eq:WeakFormulationGalerkinElectric}, we obtain
\begin{equation*}
\begin{split}
&\phantom{=}\frac{\varepsilon_0}{2}\iint_{\mathbb{R}^3}\lVert\bm{e}_{n,k,\eta}(T,\cdot)\rVert^2
-\iint_{\rbrack0,T\lbrack\cart\mathbb{R}^3}\bm{e}_{n,k,\eta}\cdot\Rot\bm{h}_{n,k,\eta}\ud\bm{x}\ud
t
\\&\phantom{\leq}
+\sigma\iint_{\rbrack0,T\lbrack\cart\mathbb{R}^3}\lVert\bm{e}_{n,k,\eta}\rVert^2\ud\bm{x}\ud t
+\sigma\iint_{\rbrack0,T\lbrack\cart\mathbb{R}^3}\bm{f}\cdot\bm{e}_{n,k,\eta}\ud\bm{x}\ud t
\\&\leq 
\frac{\varepsilon_0}{2}\iint_{\mathbb{R}^3}\lVert\mathcal{P}_{W_N}(\bm{e}_0)\rVert^2\ud\bm{x}.
\end{split}
\end{equation*}

Combining these three inequalities, we get an energy inequality
\begin{equation}\label{eq:EnergyInequalityGalerkin}
\begin{split}
&\phantom{=}\frac{A}{2}\int_{\Omega}\lVert\nabla\bm{m}_{n,k,\eta}(T,\cdot)\rVert^2\ud\bm{x}
+\frac{1}{2}\int_{\Omega}(\mathbf{K}(\bm{x})\bm{m}_{n,k,\eta}(T,\bm{x}))\cdot\bm{m}_{n,k,\eta}(T,\bm{x})\ud\bm{x}
\\&\phantom{=}
+\frac{k}{4}\int_{\Omega}(\lVert\bm{m}_{n,k,\eta}(T,\bm{x}))\rVert^2-1)^2\ud\bm{x}
\\&\phantom{=}
+\frac{\varepsilon_0}{2\mu_0}\int_{\mathbb{R}^3}\lVert\bm{e}_{n,k,\eta}(T,\bm{x})\rVert^2\ud\bm{x}
+\frac{1}{2}\int_{\mathbb{R}^3}\lVert\bm{h}_{n,k,\eta}(T,\bm{x})\rVert^2\ud\bm{x}
\\&\phantom{=}
+\mathrm{E}_s^\eta(\bm{m}_{n,k,\eta}(T,\cdot))
+\frac{\alpha}{1+\alpha^2}\iint_{Q_T}\left\lVert\frac{\partial\bm{m}_{n,k,\eta}}{\partial
    t}\right\rVert^2 \ud\bm{x}\ud t\\
&\phantom{\leq}+\frac{\sigma}{\mu_0}\iint_{\rbrack0,T\lbrack\cart\mathbb{R}^3}\lVert\bm{e}_{n,k,\eta}\rVert^2\ud\bm{x}\ud t
+\frac{\sigma}{\mu_0}\iint_{\rbrack0,T\lbrack\cart\mathbb{R}^3}\bm{f}\cdot\bm{e}_{n,k,\eta}\ud\bm{x}\ud t
\\&\leq
\frac{A}{2}\int_{\Omega}\lVert\nabla\mathcal{P}_{V_n}(\bm{m}_0)\rVert^2\ud\bm{x}
+\frac{1}{2}\int_{\Omega}(\mathbf{K}(\bm{x}) \mathcal{P}_{V_n}(\bm{m}_0))\cdot\mathcal{P}_{V_n}(\bm{m}_0)\ud\bm{x}
\\&\phantom{=}
+\frac{k}{4}\int_{\Omega}(\lVert \mathcal{P}_{V_n}(\bm{m}_0)\rVert^2-1)^2\ud\bm{x}
+\mathrm{E}_s^\eta(\mathcal{P}_{V_n}(\bm{m}_0))
\\&\phantom{=}
+\frac{\varepsilon_0}{2\mu_0}\int_{\mathbb{R}^3}\lVert\mathcal{P}_{W_N}(\bm{e}_0)\rVert^2\ud\bm{x}
+\frac{1}{2}\int_{\mathbb{R}^3}\lVert\mathcal{P}_{W_N}(\bm{h}_0)\rVert^2\ud\bm{x}
\end{split}
\end{equation}
The projection $\mathcal{P}_n(\bm{m}_0)$ converges to $\bm{m}_0$ 
in $\mathbb{H}^1(\Omega)$ and in $\mathbb{L}^6(\Omega)$
by Sobolev imbedding.
The terms on the right hand-side remain bounded independently of
$n$. The last term on the left hand-side may be dealt with by Young
inequality. Thus,  $\bm{m}_{n,k,\eta}$, $\bm{h}_{n,k,\eta}$ and $\bm{e}_{n,k,\eta}$  
cannot explode in finite time and exist globally. 

\subsection{Final step of Galerkin's method}\label{subsect:nLimit}
We now  have $n$ tend to $+\infty$
By~\eqref{eq:EnergyInequalityGalerkin} and using Young inequality to
deal with the term containing $\bm{f}$:
\begin{itemize}
\item $\bm{m}_{n,k,\eta}$  is bounded in $\mathrm{L}^\infty(\mathbb{R}^+;\mathbb{L}^4(\Omega))$ independently
of $n$.
\item $\nabla \bm{m}_{n,k,\eta}$ is bounded in $\mathrm{L}^\infty(\mathbb{R}^+;\mathbb{L}^2(\Omega))$ independently
of $n$.
\item $\frac{\partial\bm{m}_{n,k,\eta}}{\partial t}$ is bounded in $\mathrm{L}^2(\mathbb{R}^+;\mathbb{L}^2(\Omega))$ independently
of $n$.
\item $\bm{h}_{n,k,\eta}$  is bounded in $\mathrm{L}^\infty(\mathbb{R}^+;\mathbb{L}^2(\Omega))$ independently
of $n$.
\item $\bm{e}_{n,k,\eta}$  is bounded in $\mathrm{L}^\infty(\mathbb{R}^+;\mathbb{L}^2(\Omega))$ independently
of $n$.
\end{itemize}
Thus, there exist $\bm{m}_{k,\eta}$ in
$\mathrm{H}_{loc}^1(\lbrack0,+\infty\lbrack;\mathbb{L}^2(\Omega))\cap
\mathrm{L}^\infty(0,+\infty;\mathbb{H}^1(\Omega))$, 
$\bm{h}_{k,\eta}$ in $\mathrm{L}^\infty(\mathbb{R}^+;\mathbb{L}^2(\Omega))$, 
$\bm{e}_{k,\eta}$ in $\mathrm{L}^\infty(\mathbb{R}^+;\mathbb{L}^2(\Omega))$, 
such that up to a subsequence:
\begin{itemize}
\item $\bm{m}_{n,k,\eta}$ converges weakly to $\bm{m}_{k,\eta}$ in
 $\mathbb{H}^1(\rbrack0,T\lbrack\cart\Omega)$.
\item $\bm{m}_{n,k,\eta}$ converges strongly to $\bm{m}_{k,\eta}$ in
$\mathbb{L}^2(\rbrack0,T\lbrack\cart\Omega)$. 
\item $\bm{m}_{n,k,\eta}$ converges strongly to $\bm{m}_{k,\eta}$ in
$\mathcal{C}(\lbrack0,T\rbrack;\mathbb{L}^2(\Omega))$ and thus in
$\mathcal{C}(\lbrack0,T\rbrack;\mathbb{L}^p(\Omega))$ for all $1\leq p<6$.
\item $\nabla\bm{m}_{n,k,\eta}$ converges weakly to $\nabla\bm{m}_{k,\eta}$ in
$\mathbb{L}^2(\rbrack0,T\lbrack\cart\Omega)$.
\item For all time $T$, $\nabla\bm{m}_{n,k,\eta}(T,\cdot)$ converges
  weakly to
$\nabla \bm{m}_{k,\eta}(T,\cdot)$ in
$\mathbb{L}^2(\Omega)$. The same subsequence can be used for all time
$T\geq 0$, see Lemma~\ref{lemma:SameSubsequenceGivenTime}.
\item $\frac{\partial\bm{m}_{n,k,\eta}}{\partial t}$ converges star weakly to
  $\frac{\partial\bm{m}_{k,\eta}}{\partial t}$ in
$\mathrm{L}^\infty(\mathbb{R}^+;\mathbb{L}^2(\Omega))$.
\item $\bm{h}_{n,k,\eta}$ converges star weakly to $\bm{h}_{k,\eta}$ in $\mathrm{L}^\infty(\mathbb{R}^+;\mathbb{L}^2(\Omega))$.
\item $\bm{e}_{n,k,\eta}$ converges star weakly to $\bm{e}_{k,\eta}$ in $\mathrm{L}^\infty(\mathbb{R}^+;\mathbb{L}^2(\Omega))$.
\end{itemize}
Moreover, by Aubin's lemma, see~\cite{Aubin:TheoremeCompacite}, $\bm{m}_{n,k,\eta}$ converges strongly to 
$\bm{m}_{k,\eta}$ in $\mathrm{L}^p(\mathbb{R}^+;\mathrm{L}^q(\Omega))$
for $1\leq p<+\infty$ and $1\leq q <6$.

Taking the limit in the energy
inequality~\eqref{eq:EnergyInequalityGalerkin} as $n$ tend to
$+\infty$ is tricky:  the terms
involving the $\mathbb{L}^2(\Omega)$ norm of $\bm{e}_{n,k,\eta}(T,\cdot)$ and
$\bm{h}_{n,k,\eta}(T,\cdot)$ are tricky. For all $T>0$, we can extract
a subsequence of $\bm{e}_{n,k,\eta}(T,\cdot)$ that converges weakly to
$\bm{e}^T_{k,\eta}$  in $L^2(\Omega)$ as $n$ tends to $+\infty$. The
tricky part is that it is unproven
that $\bm{e}^T_{k,\eta}$ is equal to $\bm{e}_{k,\eta}(T,\cdot)$. If we
had strong convergence of $\bm{e}_{n,k,\eta}$ as a function defined on
$\mathbb{R}^+\cart\Omega$ or if we had the
existence of a subsequence along which $\bm{e}_{n,k,\eta}(T,\cdot)$
converged weakly in $L^2(\Omega)$ for almost all time $T$, then we
could conclude directly.  Unfortunately, while we have for all $T>0$,
the existence of a subsequence of $\bm{e}_{n,k,\eta}(T,\cdot)$ that
converges weakly in $L^2(\Omega)$, the subsequence depends on $T$. 
We have the same problem for $\bm{h}_{n,k,\eta}$.  There's no such
problem with $\bm{m}(T,\cdot)$,
see Lemma~\ref{lemma:SameSubsequenceGivenTime}. To solve the problem, 
we first integrate ~\eqref{eq:EnergyInequalityGalerkin} over $\rbrack
T_1,T_2\lbrack$ where $0\leq T_1<T_2<+\infty$ then we can
take the limit as $n$ tend to $+\infty$:
\begin{equation*}
\begin{split}
&\phantom{=}\frac{A}{2}\int_{T_1}^{T_2}\int_{\Omega}\lVert\nabla\bm{m}_{k,\eta}(T,\cdot)\rVert^2\ud\bm{x}\ud T
+\frac{1}{2}\int_{T_1}^{T_2}\int_{\Omega}(\mathbf{K}(\bm{x})\bm{m}_{k,\eta}(T,\bm{x}))\cdot\bm{m}_{k,\eta}(T,\bm{x})\ud\bm{x}\ud T
\\&\phantom{=}
+\frac{k}{4}\int_{T_1}^{T_2}\int_{\Omega}(\lVert\bm{m}_{k,\eta}(T,\bm{x}))\rVert^2-1)^2\ud\bm{x}
\\&\phantom{=}
+\frac{\varepsilon_0}{2\mu_0}\int_{T_1}^{T_2}\int_{\mathbb{R}^3}\lVert\bm{e}_{k,\eta}(T,\bm{x})\rVert^2\ud\bm{x}
+\frac{1}{2}\int_{T_1}^{T_2}\int_{\mathbb{R}^3}\lVert\bm{h}_{k,\eta}(T,\bm{x})\rVert^2\ud\bm{x}
\\&\phantom{=}
+\int_{T_1}^{T_2}\mathrm{E}_s^\eta(\bm{m}_{k,\eta}(T,\cdot))
+\frac{\alpha}{1+\alpha^2}\int_{T_1}^{T_2}\iint_{Q_T}\left\lVert\frac{\partial\bm{m}_{k,\eta}}{\partial
    t}\right\rVert^2 \ud\bm{x}\ud t\\
&\phantom{\leq}+\frac{\sigma}{\mu_0}\int_{T_1}^{T_2}\iint_{\rbrack0,T\lbrack\cart\mathbb{R}^3}\lVert\bm{e}_{k,\eta}\rVert^2\ud\bm{x}\ud t
+\frac{\sigma}{\mu_0}\int_{T_1}^{T_2}\iint_{\rbrack0,T\lbrack\cart\mathbb{R}^3}\bm{f}\cdot\bm{e}_{k,\eta}\ud\bm{x}\ud t
\\&\leq
\frac{A}{2}\int_{T_1}^{T_2}\int_{\Omega}\lVert\nabla\bm{m}_0\rVert^2\ud\bm{x}
+\frac{1}{2}\int_{T_1}^{T_2}\int_{\Omega}(\mathbf{K}(\bm{x}) \bm{m}_0)\cdot\bm{m}_0\ud\bm{x}
\\&\phantom{=}
+\int_{T_1}^{T_2}\mathrm{E}_s^\eta(\bm{m}_0)
+\frac{\varepsilon_0}{2\mu_0}\int_{T_1}^{T_2}\int_{\mathbb{R}^3}\lVert\bm{e}_0\rVert^2\ud\bm{x}
+\frac{1}{2}\int_{T_1}^{T_2}\int_{\mathbb{R}^3}\lVert\bm{h}_0\rVert^2\ud\bm{x},
\end{split}
\end{equation*}
for all $0\leq T_1<T_2<+\infty$. Since the equality holds for all
$T_1$ and $T_2$, we have for almost all $T>0$
\begin{equation}\label{eq:EnergyInequalityPenalized}
\begin{split}
&\phantom{=}\frac{A}{2}\int_{\Omega}\lVert\nabla\bm{m}_{k,\eta}(T,\cdot)\rVert^2\ud\bm{x}\ud T
+\frac{1}{2}\int_{\Omega}(\mathbf{K}(\bm{x})\bm{m}_{k,\eta}(T,\bm{x}))\cdot\bm{m}_{k,\eta}(T,\bm{x})\ud\bm{x}\ud T
\\&\phantom{=}
+\frac{k}{4}\int_{\Omega}(\lVert\bm{m}_{k,\eta}(T,\bm{x}))\rVert^2-1)^2\ud\bm{x}
\\&\phantom{=}
+\frac{\varepsilon_0}{2\mu_0}\int_{\mathbb{R}^3}\lVert\bm{e}_{k,\eta}(T,\bm{x})\rVert^2\ud\bm{x}
+\frac{1}{2}\int_{\mathbb{R}^3}\lVert\bm{h}_{k,\eta}(T,\bm{x})\rVert^2\ud\bm{x}
\\&\phantom{=}
+\mathrm{E}_s^\eta(\bm{m}_{k,\eta}(T,\cdot))
+\frac{\alpha}{1+\alpha^2}\iint_{Q_T}\left\lVert\frac{\partial\bm{m}_{k,\eta}}{\partial
    t}\right\rVert^2 \ud\bm{x}\ud t\\
&\phantom{\leq}+\frac{\sigma}{\mu_0}\iint_{\rbrack0,T\lbrack\cart\mathbb{R}^3}\lVert\bm{e}_{k,\eta}\rVert^2\ud\bm{x}\ud t
+\frac{\sigma}{\mu_0}\iint_{\rbrack0,T\lbrack\cart\mathbb{R}^3}\bm{f}\cdot\bm{e}_{k,\eta}\ud\bm{x}\ud t
\\&\leq
\frac{A}{2}\int_{\Omega}\lVert\nabla\bm{m}_0\rVert^2\ud\bm{x}
+\frac{1}{2}\int_{\Omega}(\mathbf{K}(\bm{x}) \bm{m}_0)\cdot\bm{m}_0\ud\bm{x}
\\&\phantom{=}
+\mathrm{E}_s^\eta(\bm{m}_0)
+\frac{\varepsilon_0}{2\mu_0}\int_{\mathbb{R}^3}\lVert\bm{e}_0\rVert^2\ud\bm{x}
+\frac{1}{2}\int_{\mathbb{R}^3}\lVert\bm{h}_0\rVert^2\ud\bm{x},
\end{split}
\end{equation}

We take the limit in~\eqref{eq:WeakFormulationGalerkinMagnetization}
as $n$ tends to $+\infty$:
\begin{subequations}\label{subeq:WeakFormulationPenalized}
\begin{equation}\label{eq:WeakFormulationPenalizedMagnetization}
\begin{split}
&\phantom{=}
\iint_{Q_T}
\alpha\frac{\partial\bm{m}_{k,\eta}}{\partial t}\cdot\bm{\phi}\ud\bm{x}\ud t
+\iint_{Q_T}
	\left(\bm{m}_{k,\eta}\vect\frac{\partial\bm{m}_{k,\eta}}{\partial t}\right)
	\cdot\bm{\phi}\ud\bm{x}\ud t 
\\&=
-(1+\alpha^2) A\iint_{Q_T}\sum_{i=1}^3
	\frac{\partial \bm{m}_{k,\eta}}{\partial x_i}\cdot\frac{\partial\bm{\phi}}{\partial x_i}\ud\bm{x}\ud t
 \\&\phantom{=}
-(1+\alpha^2)\iint_{Q_T}
	(\mathbf{K}(\bm{x}) \bm{m}_{k,\eta}(t,\bm{x}))
	\cdot\bm{\phi}(t,\bm{x})\ud\bm{x}\ud t
\\&\phantom{=}
+(1+\alpha^2)\iint_{Q_T}
	\bm{h}_{k,\eta}
	\cdot\bm{\phi}\ud\bm{x}\ud t					
\\&\phantom{=}
        +(1+\alpha^2)\frac{K_s}{\eta}\iint_{\rbrack0,T\lbrack\cart(B\cart\rbrack-\eta,\eta\lbrack)}
        ((\bm{\nu}\cdot\bm{m}_{k,\eta})\bm{\nu}-\bm{m}_{k,\eta})
	\cdot\bm{\phi}\ud\bm{x}\ud t					
\\&\phantom{=}
+(1+\alpha^2) \frac{ J_1}{\eta}\iint_{\rbrack0,T\lbrack\cart(B\cart\rbrack-\eta,\eta\lbrack)}
	(\bm{m}^{*}_{k,\eta}-\bm{m}_{k,\eta})\cdot\bm{\phi}\ud\bm{x}\ud t		
\\&\phantom{=}
	+2(1+\alpha^2)\frac{J_2}{\eta}
		\iint_{\rbrack0,T\lbrack\cart(B\cart\rbrack-\eta,\eta\lbrack)}
	\left((\bm{m}_{k,\eta}\cdot\bm{m}_{k,\eta}^{*})\bm{m}^{*}_{n,k,\eta}
        -\lVert\bm{m}_{k,\eta}^*\rVert^2\bm{m}_{k,\eta}\right)
	\cdot\bm{\phi}\ud\bm{x}\ud t,
\end{split} 
\end{equation}
for all $\bm{\phi}$ in $\bigcup_n\mathcal{C}^\infty(\lbrack0,T\lbrack;V_n^3)$. By
density, it also holds for all $\bm{\phi}$ in
$\mathbb{H}^1(\rbrack0,T\lbrack\cart\Omega)$.
We integrate~\eqref{eq:WeakFormulationGalerkinExcitation} by parts then take the limit as
$n$ tends to $+\infty$. 
\begin{equation}\label{eq:WeakFormulationPenalizedExcitation}
\begin{split} 
&\phantom{=}-\mathrm{\mu}_0\iint_{\mathbb{R}^+\cart\mathbb{R}^3}\left(\bm{h}_{k,\eta}+
\bm{m}_{k,\eta}\right)) \frac{\partial\bm{\psi}}{\partial t}\ud\bm{x}\ud t
+\iint_{\mathbb{R}^+\cart\mathbb{R}^3}\bm{e}_{k,\eta}\cdot\Rot\bm{\psi}\ud\bm{x}\ud
t\\
&=\mathrm{\mu}_0\int_{\mathbb{R}^3}\left(\bm{h}_{0}+
\bm{m}_{0}\right))\cdot\bm{\psi}(0,\cdot)\ud\bm{x},
\end{split}
\end{equation}
for all $\bm{\psi}$ in $\bigcup_n\mathcal{C}_c^\infty([0,+\infty\lbrack;W_n)$. By 
density, it also holds for all $\bm{\psi}$ in
$\mathrm{L}^1(\mathbb{R}^+;\mathbb{H}^1(\Omega))$ such that
$\frac{\partial \bm{\psi}}{\partial t}$ belongs to $\mathrm{L}^1(\mathbb{R}^+;\mathbb{H}^1(\Omega))$.
We integrate~\eqref{eq:WeakFormulationGalerkinElectric} by parts then take the limit as
$n$ tends to $+\infty$. 
\begin{equation}\label{eq:WeakFormulationPenalizedElectric}
\begin{split} 
&\phantom{=}
-\mathrm{\varepsilon}_0\iint_{\mathbb{R}^+\cart\mathbb{R}^3}\bm{e}_{k,\eta}\cdot \frac{\partial\bm{\Theta}}{\partial
    t}\ud\bm{x}\ud t-\iint_{\mathbb{R}^+\cart\mathbb{R}^3}\bm{h}_{k,\eta}\cdot \Rot\bm{\Theta}\ud\bm{x}\ud
  t
\\&\phantom{=}
+\sigma\iint_{\mathbb{R}^+\cart\Omega}(\bm{e}_{k,\eta}+\bm{f})\cdot\bm{\Theta}\ud\bm{x}\ud t
\\&=
\mathrm{\varepsilon}_0\int_{\mathbb{R}^3}\bm{e}_{0}\cdot\bm{\Theta}(0,\cdot)\ud\bm{x},
\end{split}
\end{equation}
for all $\bm{\Theta}$ in $\bigcup_n\mathcal{C}_c^\infty([0,+\infty\lbrack;W_n)$. By
density, it also holds for all $\bm{\Theta}$ in
$\mathrm{L}^1(\mathbb{R}^+;\mathbb{H}^1(\Omega))$ such that
$\frac{\partial \bm{\Theta}}{\partial t}$ belongs to $\mathrm{L}^1(\mathbb{R}^+;\mathbb{H}^1(\Omega))$.
\end{subequations}

\subsection{Limit as $k$ tends to $+\infty$}\label{subsect:kLimit}
By~\eqref{eq:EnergyInequalityPenalized} and using Young inequality to
deal with the term containing $\bm{f}$:
\begin{itemize}
\item $\bm{m}_{k,\eta}$  is bounded in in $\mathrm{L}^\infty(\mathbb{R}^+;\mathbb{L}^4(\Omega))$ independently
of $n$.
\item $\nabla \bm{m}_{k,\eta}$ is bounded in $\mathrm{L}^\infty(\mathbb{R}^+;\mathbb{L}^2(\Omega))$ independently
of $n$.
\item $\frac{\partial\bm{m}_{k,\eta}}{\partial t}$ is bounded in $\mathrm{L}^2(\mathbb{R}^+;\mathrm{L}^2(\Omega))$ independently
of $n$.
\item $\bm{h}_{k,\eta}$  is bounded in in $\mathrm{L}^\infty(\mathbb{R}^+;\mathbb{L}^2(\Omega))$ independently
of $n$.
\item $\bm{e}_{k,\eta}$  is bounded in in $\mathrm{L}^\infty(\mathbb{R}^+;\mathbb{L}^2(\Omega))$ independently
of $n$.
\item $k(\lVert\bm{m}_{k,\eta}\rVert^2-1)$  is bounded in in $\mathrm{L}^\infty(\mathbb{R}^+;\mathbb{L}^2(\Omega))$ independently
of $n$.
\end{itemize}
Thus, there exist $\bm{m}_{\eta}$, $\bm{h}_{\eta}$, $\bm{e}_{\eta}$, 
such that up to a subsequence:
\begin{itemize}
\item $\bm{m}_{k,\eta}$ converges weakly to $\bm{m}_{\eta}$ in
 $\mathbb{H}^1(\rbrack0,T\lbrack\cart\Omega)$.
\item $\bm{m}_{k,\eta}$ converges strongly to $\bm{m}_{\eta}$ in
$\mathbb{L}^2(\rbrack0,T\lbrack\cart\Omega)$. 
\item $\bm{m}_{k,\eta}$ converges strongly to $\bm{m}_{\eta}$ in
$\mathcal{C}(\lbrack0,T\rbrack;\mathbb{L}^2(\Omega))$ and thus in
$\mathcal{C}(\lbrack0,T\rbrack;\mathbb{L}^p(\Omega))$ for all $1\leq p<6$.
\item $\nabla\bm{m}_{k,\eta}$ converges weakly to $\nabla\bm{m}_{\eta}$ in
$\mathbb{L}^2(\rbrack0,T\lbrack\cart\Omega)$.
\item For all time $T$, $\nabla\bm{m}_{k,\eta}(T,\cdot)$ converges weakly to $\nabla\bm{m}_{\eta}(t,\cdot)$ in
$\mathbb{L}^2(\Omega)$.
\item $\frac{\partial\bm{m}_{k,\eta}}{\partial t}$ converges star weakly to
  $\frac{\partial\bm{m}_{\eta}}{\partial t}$ in $\mathbb{L}^\infty(\mathbb{R}^+;\mathbb{L}^2(\Omega))$.
\item $\bm{h}_{k,\eta}$ converges star weakly to $\bm{h}_{\eta}$ in $\mathrm{L}^\infty(\mathbb{R}^+;\mathbb{L}^2(\Omega))$.
\item $\bm{e}_{k,\eta}$ converges star weakly to $\bm{e}_{\eta}$ in $\mathrm{L}^\infty(\mathbb{R}^+;\mathbb{L}^2(\Omega))$.
\end{itemize}
Moreover, by Aubin's lemma $\bm{m}_{\eta}$ converges strongly to 
$\bm{m}_{\eta}$ in $\mathrm{L}^p(\mathbb{R}^+;\mathbb{L}^q(\Omega))$
for $1\leq q<+\infty$ and $1\leq q <6$. 
Since $\lVert\bm{m}_{k,\eta}\rVert^2-1$ converges to $0$, therefore 
$\lVert\bm{m}_\eta\rVert=1$ almost everywhere on $\mathbb{R}^+\cart\Omega$.

For the reasons explained in \S\ref{subsect:nLimit},
we integrate~\eqref{eq:EnergyInequalityPenalized} over $\lbrack T_1,T_2\rbrack$, drop the term
$k\lVert\lVert\bm{m}_{\eta}\rVert^2-1\rVert_{\mathrm{L}^2}^2/4$, and
compute the limit as $k$ tends to $+\infty$. After the limit is taken,
we drop the integral over $\lbrack T_1,T_2\rbrack$ and obtain that for
almost all $T>0$:
\begin{equation}\label{eq:EnergyInequalityLayer}
\begin{split}
&\phantom{=}\frac{A}{2}\int_{\Omega}\lVert\nabla\bm{m}_{\eta}(T,\cdot)\rVert^2\ud\bm{x}
+\frac{1}{2}\int_{\Omega}(\mathbf{K}(\bm{x})\bm{m}_{\eta}(T,\bm{x}))\cdot\bm{m}_{\eta}(T,\bm{x})\ud\bm{x}
\\&\phantom{=}
+\frac{\varepsilon_0}{2\mu_0}\int_{\mathbb{R}^3}\lVert\bm{e}_{\eta}(T,\bm{x})\rVert^2\ud\bm{x}
+\frac{1}{2}\int_{\mathbb{R}^3}\lVert\bm{h}_{\eta}(T,\bm{x})\rVert^2\ud\bm{x}
\\&\phantom{=}
+\mathrm{E}_s^\eta(\bm{m}_\eta(T,\cdot))
+\frac{\alpha}{1+\alpha^2}\iint_{Q_T}\left\lVert\frac{\partial\bm{m}_{\eta}}{\partial
    t}\right\rVert^2 \ud\bm{x}\ud t\\
&\phantom{\leq}+\frac{\sigma}{\mu_0}\iint_{\rbrack0,T\lbrack\cart\mathbb{R}^3}\lVert\bm{e}_{\eta}\rVert^2\ud\bm{x}\ud t
+\frac{\sigma}{\mu_0}\iint_{\rbrack0,T\lbrack\cart\mathbb{R}^3}\bm{f}\cdot\bm{e}_{\eta}\ud\bm{x}\ud t
\\&\leq
\frac{A}{2}\int_{\Omega}\lVert\nabla\bm{m}_0\rVert^2\ud\bm{x}
+\frac{1}{2}\int_{\Omega}(\mathbf{K}(\bm{x}) \bm{m}_0)\cdot\bm{m}_0\ud\bm{x}
\\&\phantom{=}
+\mathrm{E}_s^\eta(\bm{m}_0)
+\frac{\varepsilon_0}{2\mu_0}\int_{\mathbb{R}^3}\lVert\bm{e}_0\rVert^2\ud\bm{x}
+\frac{1}{2}\int_{\mathbb{R}^3}\lVert\bm{h}_0\rVert^2\ud\bm{x}. 
\end{split}
\end{equation}

We replace $\bm{\phi}$
in~\eqref{eq:WeakFormulationPenalizedMagnetization}
with $\bm{m}_{k,\eta}\vect\bm{\varphi}$ where $\bm{\varphi}$ 
is $\mathcal{C}^\infty_c(\mathbb{R}^+\cart\Omega)$:
\begin{equation*}
\begin{split}
 &\phantom{=}
 -\alpha\iint_{Q_T}
\left(\bm{m}_{k,\eta}\vect\frac{\partial\bm{m}_{k,\eta}}{\partial  t}\right)\cdot\bm{\varphi}\ud\bm{x}\ud t
+\iint_{Q_T}
	\lVert\bm{m}_{k,\eta}\rVert^2\frac{\partial\bm{m}_{k,\eta}}{\partial t}\cdot\bm{\varphi}\ud\bm{x}\ud t 
\\&=\iint_{Q_T}
	\left(\bm{m}_{k,\eta}\cdot\frac{\partial\bm{m}_{k,\eta}}{\partial t}\right)(\bm{m}_{k,\eta}\cdot\bm{\varphi})\ud\bm{x}\ud t 
\\&\phantom{=}
+(1+\alpha^2) A\iint_{Q_T}\sum_{i=1}^3
	\left(\bm{m}_{k,\eta}\vect\frac{\partial \bm{m}_{k,\eta}}{\partial x_i}\right)
        \cdot\frac{\partial\bm{\varphi}}{\partial x_i}\ud\bm{x}\ud t
\\&\phantom{=}
+(1+\alpha^2)\iint_{Q_T}
	\left(\bm{m}_{k,\eta}(t,\bm{x})\vect\mathbf{K}(\bm{x}) \bm{m}_{k,\eta}(t,\bm{x})\right)
	\cdot\bm{\varphi}(t,\bm{x})\ud\bm{x}\ud t
\\&\phantom{=}
-(1+\alpha^2)\iint_{Q_T}
	\left(\bm{m}_{k,\eta}\vect\bm{h}_{k,\eta}\right)
	\cdot\bm{\varphi}\ud\bm{x}\ud t					
\\&\phantom{=}
        -(1+\alpha^2)\frac{K_s}{\eta}\iint_{\rbrack0,T\lbrack\cart(B\cart\rbrack-\eta,\eta\lbrack)}
        (\bm{\nu}\cdot\bm{m}_{k,\eta})(\bm{m}_{k,\eta}\vect\bm{\nu})
	\cdot\bm{\varphi}\ud\bm{x}\ud t					
\\&\phantom{=}
-(1+\alpha^2) \frac{ J_1}{\eta}\iint_{\rbrack0,T\lbrack\cart(B\cart\rbrack-\eta,\eta\lbrack)}
	(\bm{m}_{k,\eta}\vect \bm{m}^{*}_{k,\eta})\cdot\bm{\varphi}\ud\bm{x}\ud t		
\\&\phantom{=}
	-2(1+\alpha^2)\frac{J_2}{\eta}
		\iint_{\rbrack0,T\lbrack\cart(B\cart\rbrack-\eta,\eta\lbrack)}
	(\bm{m}_{k,\eta}\cdot\bm{m}_{k,\eta}^{*})(\bm{m}_{k,\eta}\vect\bm{m}^{*}_{k,\eta})
	\cdot\bm{\varphi}\ud\bm{x}\ud t,
\end{split} 
\end{equation*}

We then take the
limit as $k$ tends to $+\infty$:
\begin{subequations}
\begin{equation}\label{eq:WeakFormulationLayerMagnetization}
\begin{split}
&\phantom{=}
 -\alpha\iint_{Q_T}
\left(\bm{m}_{\eta}\vect\frac{\partial\bm{m}_{\eta}}{\partial  t}\right)\cdot\bm{\varphi}\ud\bm{x}\ud t
+\iint_{Q_T}
	\frac{\partial\bm{m}_{\eta}}{\partial t}\cdot\bm{\varphi}\ud\bm{x}\ud t 
\\&=
+(1+\alpha^2) A\iint_{Q_T}\sum_{i=1}^3
	\left(\bm{m}_{\eta}\vect\frac{\partial \bm{m}_{\eta}}{\partial x_i}\right)
        \cdot\frac{\partial\bm{\varphi}}{\partial x_i}\ud\bm{x}\ud t
\\&\phantom{=}
+(1+\alpha^2)\iint_{Q_T}
	\left(\bm{m}_{\eta}(t,\bm{x})\vect\mathbf{K}(\bm{x}) \bm{m}_{\eta}(t,\bm{x})\right)
	\cdot\bm{\varphi}(t,\bm{x})\ud\bm{x}\ud t
\\&\phantom{=}
-(1+\alpha^2)\iint_{Q_T}
	\left(\bm{m}_{\eta}\vect\bm{h}_{\eta}\right)
	\cdot\bm{\varphi}\ud\bm{x}\ud t					
\\&\phantom{=}
        -(1+\alpha^2)\frac{K_s}{\eta}\iint_{\rbrack0,T\lbrack\cart(B\cart\rbrack-\eta,\eta\lbrack)}
        (\bm{\nu}\cdot\bm{m}_{\eta})(\bm{m}_{\eta}\vect\bm{\nu})
	\cdot\bm{\varphi}\ud\bm{x}\ud t					
\\&\phantom{=}
-(1+\alpha^2) \frac{ J_1}{\eta}\iint_{\rbrack0,T\lbrack\cart(B\cart\rbrack-\eta,\eta\lbrack)}
	(\bm{m}_{\eta}\vect \bm{m}^{*}_{\eta})\cdot\bm{\varphi}\ud\bm{x}\ud t		
\\&\phantom{=}
	-2(1+\alpha^2)\frac{J_2}{\eta}
		\iint_{\rbrack0,T\lbrack\cart(B\cart\rbrack-\eta,\eta\lbrack)}
	(\bm{m}_{\eta}\cdot\bm{m}_{\eta}^{*})(\bm{m}_{\eta}\vect\bm{m}^{*}_{\eta})
	\cdot\bm{\varphi}\ud\bm{x}\ud t,
\end{split} 
\end{equation}

We take the limit in~\eqref{eq:WeakFormulationPenalizedExcitation} as
$k$ tends to $+\infty$:
\begin{equation}\label{eq:WeakFormulationLayerExcitation}
\begin{split} 
&\phantom{=}-\mathrm{\mu}_0\iint_{\mathbb{R}^+\cart\mathbb{R}^3}\left(\bm{h}_{\eta}+
\bm{m}_{\eta}\right)) \frac{\partial\bm{\psi}}{\partial t}\ud\bm{x}\ud t
+\iint_{\mathbb{R}^+\cart\mathbb{R}^3}\bm{e}_{\eta}\Rot\bm{\psi}\ud\bm{x}\ud
t\\
&=\mathrm{\mu}_0\int_{\mathbb{R}^3}\left(\bm{h}_{0}+
\bm{m}_{0}\right))\cdot\bm{\psi}(0,\cdot)\ud\bm{x}
\end{split}
\end{equation}
for all $\bm{\psi}$ in
$\mathrm{L}^1(\mathbb{R}^+;\mathbb{H}^1(\Omega))$ such that
$\frac{\partial \bm{\psi}}{\partial t}$ belongs to $\mathrm{L}^1(\mathbb{R}^+;\mathbb{L}^2(\Omega))$. 

We  take the limit in~\eqref{eq:WeakFormulationPenalizedElectric} as
$k$ tends to $+\infty$. 
\begin{equation}\label{eq:WeakFormulationLayerElectric}
\begin{split} 
&\phantom{=}
-\mathrm{\varepsilon}_0\iint_{\mathbb{R}^+\cart\mathbb{R}^3}\bm{e}_{\eta}\cdot \frac{\partial\bm{\Theta}}{\partial
    t}\ud\bm{x}\ud t-\iint_{\mathbb{R}^+\cart\mathbb{R}^3}\bm{h}_{\eta}\cdot \Rot\bm{\Theta}\ud\bm{x}\ud
  t
\\&\phantom{=}
+\sigma\iint_{\mathbb{R}^+\cart\Omega}(\bm{e}_\eta+\bm{f})\cdot\bm{\Theta}\ud\bm{x}\ud t
\\&=
\mathrm{\varepsilon}_0\int_{\mathbb{R}^3}\bm{e}_{0}\cdot\bm{\Theta}(0,\cdot)\ud\bm{x},
\end{split}
\end{equation}
for all $\bm{\Theta}$ in in
$\mathrm{L}^1(\mathbb{R}^+;\mathbb{H}^1(\Omega))$ such that
$\frac{\partial \bm{\Theta}}{\partial t}$ belongs to $\mathrm{L}^1(\mathbb{R}^+;\mathbb{L}^2(\Omega))$. 
\end{subequations}

\subsection{Limit as $\eta$ tends to $0$}\label{subsect:EtaLimit}
Since $\mathbb{H}^1(\Omega)$ is continuously imbedded in
$\mathcal{C}^0\big(\rbrack-\hauteursous,\hauteursur\lbrack\setminus\{0\};\mathbb{L}^4(\base)\big)$, 
$\mathrm{E}_s^\eta(\bm{m}_0)$ remains bounded independently of
$\eta$ and converges to $\mathrm{E}_s(\bm{m}_0)$ . Thus, using~\eqref{eq:EnergyInequalityLayer} and the
constraint $\lVert\bm{m}_\eta\rVert=1$ almost everywhere:
\begin{itemize}
\item $\bm{m}_{\eta}$  is bounded in
  $\mathbb{L}^\infty(\mathbb{R}^+\cart\Omega)$ by $1$.
\item $\nabla \bm{m}_{\eta}$ is bounded in $\mathrm{L}^\infty(\mathbb{R}^+;\mathbb{L}^2(\Omega))$ independently
of $\eta$.
\item $\frac{\partial\bm{m}_{k,\eta}}{\partial t}$ is bounded in $\mathrm{L}^2(\mathbb{R}^+;\mathbb{L}^2(\Omega))$ independently
of $\eta$.
\item $\bm{h}_{k,\eta}$  is bounded in in $\mathrm{L}^\infty(\mathbb{R}^+;\mathbb{L}^2(\Omega))$ independently
of $\eta$.
\item $\bm{e}_{k,\eta}$  is bounded in in $\mathrm{L}^\infty(\mathbb{R}^+;\mathbb{L}^2(\Omega))$ independently
of $\eta$.
\end{itemize}
Thus, there exists $\bm{m}$ in
$\mathbb{L}^\infty(\mathbb{R}^+;\mathbb{H}^1(\Omega))$ and in 
$\mathbb{H}^1_{\mathrm{loc}}(\lbrack0,+\infty\lbrack;\mathbb{L}^2(\Omega))$,
$\bm{h}$ in $\mathbb{L}^\infty(\mathbb{R}^+;\mathbb{L}^2(\Omega))$
and $\bm{e}$ in $\mathbb{L}^\infty(\mathbb{R}^+;\mathbb{L}^2(\Omega))$ such that up to a subsequence
\begin{itemize}
\item $\bm{m}_{\eta}$ converges weakly to $\bm{m}$ in
 $\mathbb{H}^1(\rbrack0,T\lbrack\cart\Omega)$.
\item $\bm{m}_{\eta}$ converges strongly to $\bm{m}$ in
$\mathbb{L}^2(\rbrack0,T\lbrack\cart\Omega)$. 
\item $\bm{m}_{\eta}$ converges strongly to $\bm{m}$ in
$\mathcal{C}(\lbrack0,T\rbrack;\mathbb{L}^2(\Omega))$ and thus in
$\mathcal{C}(\lbrack0,T\rbrack;\mathbb{L}^p(\Omega))$ for all $1\leq p<+\infty$.
\item $\nabla\bm{m}_{\eta}$ converges weakly to $\nabla\bm{m}$ in
$\mathbb{L}^2(\rbrack0,T\lbrack\cart\Omega)$.
\item For all time $T$, $\nabla\bm{m}_{\eta}(t,\cdot)$ converges weakly to $\nabla\bm{m}(t,\cdot)$ in
$\mathbb{L}^2(\Omega)$.
\item $\frac{\partial\bm{m}_{\eta}}{\partial t}$ converges star weakly to
  $\frac{\partial\bm{m}}{\partial t}$ in $\mathrm{L}^\infty(\mathbb{R}^+;\mathbb{L}^2(\Omega))$.
\item $\bm{h}_{\eta}$ converges star weakly to $\bm{h}$ in $\mathrm{L}^\infty(\mathbb{R}^+;\mathbb{L}^2(\Omega))$.
\item $\bm{e}_{\eta}$ converges star weakly to $\bm{e}$ in $\mathrm{L}^\infty(\mathbb{R}^+;\mathbb{L}^2(\Omega))$.
\end{itemize}
As $\lVert\bm{m}^\eta\rVert=1$ almost everywhere,
$\lVert\bm{m}\rVert=1$ almost everywhere. Moreover, as
$\bm{m}_\eta(0,\cdot)=\bm{m}_0$, we have $\bm{m}(0,\cdot)=\bm{m}_0$.

For the reasons explained in \S\ref{subsect:nLimit},
we integrate~\eqref{eq:EnergyInequalityLayer} over $\lbrack T_1,T_2\rbrack$, and
compute the limit as $k$ tends to $+\infty$. All the volume terms
converge to their intuitive limit. After the limit is taken,
we drop the integral over $\lbrack T_1,T_2\rbrack$ and obtain that for
almost all $T>0$:
Taking the limit in the surfacic terms requires more work. For
easier understanding, 

First, the space $\mathbb{H}^1(\rbrack0,T\lbrack\cart\Omega)$ is compactly imbedded into
\begin{equation*}
\mathcal{C}^0(\lbrack-\hauteursous,0\rbrack;\mathbb{L}^2(\rbrack0,T\lbrack\cart\base)\otimes\mathcal{C}^0(\lbrack0,\hauteursur\rbrack;\mathbb{L}^2(\rbrack0,T\lbrack\cart\base)).
\end{equation*}
 This
is a direct application of Lemma~\ref{lemma:compactCL2imbedding} with
$\mathcal{O}=\rbrack0,T\lbrack \cart B$
and, thus a
direct consequence of the extended Aubin's
lemma~\ref{lemma:AubinSimon}. 
Therefore, $\bm{m}_\eta$ converges
strongly to $\bm{m}$ in
\begin{equation*}
\mathcal{C}^0(\lbrack-\hauteursous,0\rbrack;\mathbb{L}^2(\rbrack0,T\lbrack\cart\base)\otimes\mathcal{C}^0(\lbrack0,\hauteursur\rbrack;\mathbb{L}^2(\rbrack0,T\lbrack\cart\base)).
\end{equation*}
Since $\lVert \bm{m}_\eta\rVert=1$, the convergence is strong in 
\begin{equation*}
\mathcal{C}^0(\lbrack-\hauteursous,0\rbrack;\mathbb{L}^p(\rbrack0,T\lbrack\cart\base)\otimes\mathcal{C}^0(\lbrack0,\hauteursur\rbrack;\mathbb{L}^p(\rbrack0,T\lbrack\cart\base)),
\end{equation*}
for all $p<+\infty$.  

\begin{equation*}
\begin{split}
&\phantom{\leq}\limsup_{\eta\to0}\lVert\int_{T_1}^{T_2}\mathrm{E}^\eta_s(\bm{m}_\eta(t,\cdot))-\mathrm{E}^\eta_s(\bm{m}(t,\cdot))\rVert
\\&\leq
\limsup_{\eta\to0}\frac{1}{2\eta}\int_{-\eta}^{\eta}\int_{T_1}^{T_2}\iint_{\base}\left\lVert
  P(\bm{m}_\eta(t),\bm{m}_\eta^{*}(t))
-P(\bm{m}(t),\bm{m}^{*}(t))\right\rVert\ud x\ud y\ud z \ud t
\\&\leq
\limsup_{\eta\to0}\sup_{z\in\lbrack-\eta,\eta\rbrack}\int_{T_1}^{T_2}\iint_{\base}\left\lVert P(\bm{m}_\eta(t),\bm{m}_\eta^{*}(t))
-P(\bm{m}(t),\bm{m}^{*}(t))
\right\rVert\ud x\ud y
\\&\leq0.
\end{split}
\end{equation*} 
where $P$ is some polynomial.

Moreover, $\bm{m}(\cdot,\cdot)$ belongs to:
\begin{equation*}
\mathcal{C}^0\big(\lbrack-\hauteursous,0\rbrack;\mathbb{L}^p(\rbrack0,T\lbrack\cart\base)\big)
\otimes
\mathcal{C}^0\big(\lbrack0,\hauteursur\rbrack;\mathbb{L}^p(\rbrack0,T\lbrack\cart\base)\big).
\end{equation*}
Therefore, we have
\begin{equation*}
\begin{split}
&\phantom{\leq}
\lim_{\eta\to0}\int_{T_1}^{T_2}\lVert\mathrm{E}^\eta_s(\bm{m}(t,\cdot))-\mathrm{E}_s(\bm{m}(t,\cdot))\rVert
\\&\leq \lim_{\eta\to0}\frac{1}{2\eta}\int_{-\eta}^{\eta}\iint_{\base}\left\lVert\big(P(\bm{m}(t),\bm{m}^{*}(t))
-P(\bm{m}(x,y,0^+,t),\bm{m}(x,y,0^-,t))\big)
\right\rVert\ud x\ud y\ud t
\\&\leq \lim_{\eta\to0}\sup_{\lVert z\rVert<\eta}\int_{T_1}^{T_2}\iint_{\base}\left\lVert P(\bm{m}(z,T),\bm{m}^{*}(z,t))
-P(\bm{m}(x,y,0^+,T),\bm{m}(x,y,0^-,t))
\right\rVert\ud x\ud y\ud t
\\&\leq 0.
\end{split}
\end{equation*}
Hence, the integral over $\lbrack T_1,T_2\rbrack$ of
inequality~\eqref{eq:EnergyInequality} hold for all $0<T_1<T_2$,
therefore inequality~\eqref{eq:EnergyInequality} is satisfied for
almost all $t>0$.

We take the limit in~\eqref{eq:WeakFormulationLayerMagnetization} as $\eta$ tends
to $0$. All the volume terms converges to their intuitive limit.
Moreover, because of the strong  convergence, along a subsequence, of
$\bm{m}_\eta$ to $\bm{m}$ in 
\begin{equation*}
\mathcal{C}^0(\lbrack-\hauteursous,0\rbrack;\mathbb{L}^p(\rbrack0,T\lbrack\cart\base)\otimes\mathcal{C}^0(\lbrack0,\hauteursur\rbrack;\mathbb{L}^p(\rbrack0,T\lbrack\cart\base)),
\end{equation*}
for all $p<+\infty$,
we have 
\begin{align*}
\begin{split} 
\limsup_{\eta\to0}  \frac{1}{\eta}\bigg\lVert
\iint_{\rbrack0,T\lbrack\cart(B\cart\rbrack-\eta,\eta\lbrack)}
        (\bm{\nu}\cdot\bm{m}_{\eta})(\bm{m}_{\eta}\vect\bm{\nu})
	\cdot\bm{\varphi}(t,\bm{x}) \ud\bm{x}\ud t&\\
-\iint_{\rbrack0,T\lbrack\cart(B\cart\rbrack-\eta,\eta\lbrack)}
        (\bm{\nu}\cdot\bm{m})(\bm{m}\vect\bm{\nu})
	\cdot\bm{\varphi}(t,\bm{x})
\ud\bm{x}\ud t
\bigg\rVert&=0,
\end{split}\\
\begin{split}
\limsup_{\eta\to0}  \frac{1}{\eta}\bigg\lVert
\iint_{\rbrack0,T\lbrack\cart(B\cart\rbrack-\eta,\eta\lbrack)}
	(\bm{m}_{\eta}\vect
       \bm{m}^{*}_{\eta})\cdot\bm{\varphi}(t,\bm{x}) \ud\bm{x}\ud t&\\
-\iint_{\rbrack0,T\lbrack\cart(B\cart\rbrack-\eta,\eta\lbrack)}
	(\bm{m}\vect
       \bm{m}^{*})\cdot\bm{\varphi}(t,\bm{x})
\ud\bm{x}\ud t
\bigg\rVert
&=0,
\end{split}\\					
\begin{split}
\limsup_{\eta\to0}  \frac{1}{\eta}\bigg\lVert
	\frac{1}{\eta}
		\iint_{\rbrack0,T\lbrack\cart(B\cart\rbrack-\eta,\eta\lbrack)}
	(\bm{m}_{\eta}\cdot\bm{m}_{\eta}^{*})(\bm{m}_{\eta}\vect\bm{m}^{*}_{k,\eta})
	\cdot\bm{\varphi}(t,\bm{x}) \ud\bm{x}\ud t&\\
-		\iint_{\rbrack0,T\lbrack\cart(B\cart\rbrack-\eta,\eta\lbrack)}
	(\bm{m}\cdot\bm{m}^{*})(\bm{m}\vect\bm{m}^{*})
	\cdot\bm{\varphi}(t,\bm{x})
\ud\bm{x}\ud t\bigg\rVert&=0.
\end{split} 
\end{align*}
Since 
$\bm{m}$ belongs to 
\begin{equation*}
\mathcal{C}^0(\lbrack-\hauteursous,0\rbrack;\mathbb{L}^p(\rbrack0,T\lbrack\cart\base)\otimes\mathcal{C}^0(\lbrack0,\hauteursur\rbrack;\mathbb{L}^p(\rbrack0,T\lbrack\cart\base)),
\end{equation*}
each surface term also converges to its surface intuitive
limits. Therefore, the weak
formulation~\eqref{eq:WeakFormulationMagnetization} is also satisfied.

We take the limits as $\eta$ tends to $0$ in~\eqref{eq:WeakFormulationLayerExcitation}
and~\eqref{eq:WeakFormulationLayerExcitation}. All the volume terms
converges to their intuitive limit.
Hence, relations~\eqref{eq:WeakFormulationExcitation} 
and~\eqref{eq:WeakFormulationElectric} are satisfied.
This finishes our proof of Theorem~\ref{theo:ExistenceWeakLandauLifshitzMaxwellSurfaceEnergies}.

%
%

\section{Characterization of the $\omega$-limit set}
\label{section-omegalimit}

We consider $(\bm{m},\bm h,\bm e)$ a weak solution to the Landau-Lifschitz-Maxwell system given by Theorem~\ref{theo:ExistenceWeakLandauLifshitzMaxwellSurfaceEnergies}.

We consider $\bm u\in \omega(m)$. There exists a non decreasing sequence $(t_n)_n$ such that $t_n\ds +\infty$, and $\bm m (t_n,.)\rightharpoonup \bm u$ in $\HH^1(\Omega)$ weak.
Since $\Omega $ is a smooth bounded domain, then $\bm m (t_n,.)$ tends to $\bm u$ in $\LL^p(\Omega)$ strongly for $p\in [1,6[$, and extracting a subsequence, we assume that $\bm m(t_n,.)$ tends to $\bm u$ almost everywhere, so that the saturation constraint $\vert \bm u \vert =1$ is satisfied almost everywhere.

In addition, we remark that for all $n$, $\vert \bm m(t_n,.)\vert=1$ almost everywhere, so that $\|\bm m(t_n,.)\|_{\LL^\infty(\Omega)}=1$. By interpolation inequalities in the $\L^p$ spaces, we obtain that for all $p<+\infty$, $\bm m(t_n,.)$ tends to $\bm u$ in $\LL^p(\Omega) $ strongly.

\vspace{3mm}
{\bf First Step.} we fix $a$ a non negative real number. for $s\in ]-a,a[$ and $x\in \Omega$, for $n$ large enough, we set 
\begin{equation*}U_n(s,x)=\bm m(t_n+s,x).\end{equation*}

We have the following estimate:
\begin{equation*}\begin{array}{rl}
\dsp \frac{1}{2a} \int_{-a}^a\int_\Omega |U_n(s,x)-\bm m (t_n,x)|^2dx ds = & \dsp \frac{1}{2a} \int_{-a}^a \int_\Omega \left| \int_0^s \dd{\bm m}{t}(t_n+\tau,x) d\tau \right|^2 dx ds\\\\
\leq &\dsp \frac{1}{2a} \int_{-a}^a|s|\int_\Omega \int_{t_n-a}^{+\infty}  \left| \dd{\bm m}{t}(\tau,x)\right|^2 d\tau dx ds\\
\\
\leq & \dsp a\int_{t_n-a}^{+\infty} \int_\Omega\left| \dd{\bm m}{t}(\tau,x)\right|^2 d\tau dx .
\end{array}\end{equation*}
Since $\dsp \dd{\bm m}{t}$ is in $\LL^2(\R^+\times \Omega)$, we obtain that 
\begin{equation*}\int_{-a}^a\int_\Omega |U_n(s,x)-\bm m (t_n,x)|^2dx ds\ds 0\mbox{ as $n$ tends to }+\infty.\end{equation*}
Since $\bm m (t_n,.)$ tends strongly  to $\bm u$ in $L^2(\Omega)$, then 
\begin{equation}
\label{cacaprout}
U_n\mbox{ tends strongly to } \bm u \mbox{ in }L^2(-a,a; \LL^2(\Omega)).\end{equation}

We remark now that the sequence $(\nabla U_n)_n$ is bounded in $L^\infty(-a,a;\LL^2(\Omega))$. In addition, $(\dd{U_n}{t})_n$ is bounded in $L^2(-a,a;\LL^2(\Omega))$. 
So, by applying Aubin's Lemma with $X=\HH^1(\Omega)$, $B=    \HH^{\frac{3}{4}}(\Omega)$, $Y=\LL^2(\Omega)$, $r=2$ and $p=+\infty$, we obtain that $(U_n)_n$ is compact in $\CC^0([-a,a];\HH^{\frac{3}{4}}(\Omega))$, so that
\begin{equation}
\label{cacaproutbis}
U_n\mbox{ tends strongly to } \bm u \mbox{ in }\CC^0([-a,a];\HH^{\frac{3}{4}}(\Omega)).\end{equation}

By continuity of the trace operator, since $\HH^{\frac{1}{4}}(\Gamma)\subset \LL^2(\Gamma)$, we obtain that $$\gamma(U_n)\longrightarrow \gamma(\bm u)\mbox{ strongly in }\CC^0([-a,a];\LL^2(\Gamma)).$$

\vspace{2mm}
In addition, by classical properties of the trace operator, for all $n$, $\|U_n\|_{L^\infty([-a,a]\times\Omega)}=1$, so $\|\gamma(U_n)\|_{L^\infty([-a,a]\times\Gamma)}\leq 1$. We obtain then in particular that\begin{equation*}
\gamma (U_n)\ds \gamma (\bm u)\mbox{ strongly in }\LL^p([-a,a]\times \partial \Omega), \; p<+\infty\end{equation*}

\vspace{2mm}
{\bf Second step.} We consider a smooth positive function $\rho_a$ compactly supported in $[-a,a]$ such that
\begin{equation*}\begin{array}{l}
\rho_a(\tau)=1\mbox{ for }\tau\in [-a+1,a-1],\\
\\
0\leq \rho_a\leq 1,\\
\\
\vert \rho_a'\vert\leq 2.
\end{array}\end{equation*}

For $n$ great enough, we set
\begin{equation*}\bm h_a^n(x)=\frac{1}{2a}\int_{-a}^a \bm h (t_n+s,x)\rho_a(s) ds \mbox{ and }\bm e_a^n(x)=\frac{1}{2a}\int_{-a}^a \bm e (t_n+s,x)\rho_a(s) ds.\end{equation*}
By construction of $(\bm m, \bm h ,\bm e)$, we know that $\bm h$ and $\bm e$ are in $\L^\infty(\mathbb R^+; \mathbb L ^2(\mathbb R ^3)).$ We have the following estimate:

\begin{equation*}
\begin{array}{rl}
\dsp \|\bm h_a^n\|_{\mathbb L^2(\mathbb R^3)}^2 = & \dsp \int_{x\in \mathbb R^3} \left|\frac{1}{2a}
\int_{-a}^a \bm h (t_n + s,x)\rho_a(s) ds\right|^2\\
\\
\leq &\dsp \frac{1}{2a}\int_{-a}^a \rho_a^2(s)ds \frac{1}{2a}  \int_{\mathbb R^3}\int_{-a}^a |\bm h (t_n+s,x)|^2ds dx\\
\\
\leq & \dsp \frac{2a+2}{2a} \|\bm h \|_{L^\infty(\mathbb R^+;\mathbb L^2(\mathbb R^3))}.
\end{array}\end{equation*}

Therefore, 
\begin{equation}
\label{estihna}
\forall a\geq 1, \; \forall n,\;  \|\bm h^n_a \|_{\mathbb L^2(\mathbb R^3)}\leq 2 \|\bm h \|_{L^\infty(\mathbb R^+;\mathbb L^2(\mathbb R^3))}.
\end{equation}

In the same way, we prove that
\begin{equation}
\label{estiena}
\forall a\geq 1, \; \forall n,\;  \|\bm e^n_a \|_{\mathbb L^2(\mathbb R^3)}\leq 2 \|\bm e \|_{L^\infty(\mathbb R^+;\mathbb L^2(\mathbb R^3))}.
\end{equation}

So for a fixed value of $a$ we can assume by extracting a subsequence that $\bm h^n_a$ and $\bm e^n_a$ converge weakly in $\mathbb L^2(\mathbb R^3)$ when $n$ tends to $+\infty$:
\begin{equation*}\bm h^n_a \rightharpoonup \bm h_a \mbox{ and }    \bm e^n_a \rightharpoonup \bm e_a\mbox{  weakly  in } \mathbb L^2(\mathbb R^3)\mbox{ when }n\rightarrow +\infty.\end{equation*} 

In the weak formulation \eqref{eq:WeakFormulationMagnetization}, we take $\bm \phi(t,x)=\frac{1}{2a}\rho_a(t-t_n)\bm \psi (x)$ where $\bm \psi\in {\mathcal D}(\overline{\Omega})$. We obtain after the change of variables $s=t-t_n$:

\begin{equation*}\frac{1}{2a} \int_{-a}^a \int_\Omega \left( \dd{U_n}{t}-\alpha U_n \vect \dd{U_n}{t}\right) \bm \psi(x) \rho_a(s) dx ds= T_1 + \ldots +T_6\end{equation*}
with
\begin{equation*}T_1=
(1+\alpha^2) A\frac{1}{2a}\int_{-a}^a \int_\Omega \sum_{i=1}^3
	\left(U_n(s,\bm{x})\vect\frac{\partial U_n}{\partial x_i}(t,\bm{x})\right)
	\cdot\frac{\partial\bm{\psi}}{\partial x_i}(\bm{x})\ud\bm{x}\ud s,\end{equation*}
	
\begin{equation*}T_2=(1+\alpha^2)\frac{1}{2a}\int_{-a}^a\int_\Omega
	\left(U_n(s,\bm{x})\vect\mathbf{K}(\bm{x})U_n(s,\bm{x})\right)		
	\cdot\bm{\psi}(\bm{x})\rho_a(s)\ud\bm{x}\ud s,\end{equation*}
	
\begin{equation*}T_3=
-(1+\alpha^2)\frac{1}{2a}\int_{-a}^a\int_\Omega
	\left(U_n(s,\bm{x})\vect\bm{h}(t_n+s,\bm{x})\right)		
	\cdot\bm{\psi}(\bm{x})\rho_a(s)\ud\bm{x}\ud s,\end{equation*}

\begin{equation*}T_4=
-(1+\alpha^2)K_s\frac{1}{2a}\int_{-a}^a
\int_{(\Gamma^\pm)}
	(\bm{\nu}\cdot\gamma U_n)(\gamma U_n\vect\bm{\nu})
	\cdot\gamma\bm{\psi}(\hat{\bm{x}})\rho_a(s)\ud S(\hat{\bm{x}})\ud s,\end{equation*}
	
\begin{equation*}T_5= 
-(1+\alpha^2)J_1\frac{1}{2a}\int_{-a}^a
\int_{(\Gamma^\pm)}
	(\gamma U_n \vect\gamma^{*}U_n)\cdot\gamma\bm{\psi}(\hat{\bm{x}})\rho_a(s)\ud S(\hat{\bm{x}})\ud s,	\end{equation*}
	
\begin{equation*}T_6=	
	-2(1+\alpha^2)J_2
		\frac{1}{2a}\int_{-a}^a
\int_{(\Gamma^\pm)}
	(\gamma U_n\cdot\gamma^{*}U_n)(\gamma U_n\vect\gamma^{*}U_n)
	\cdot\gamma\bm{\psi}(\hat{\bm{x}})\rho_a(s)\ud S(\hat{\bm{x}})\ud s.\end{equation*}

Now for a fixed value of the parameter $a$, we take the limit of the previous equation when $n$ tends to $+\infty$.

{\it Left hand side term:} we have the following estimates.

\begin{equation*}\begin{array}{r}
\dsp \left| 
\frac{1}{2a} \int_{-a}^a \int_\Omega \left( \dd{U_n}{t}-\alpha U_n \vect \dd{U_n}{t}\right) \bm \psi(x) \rho_a(s) dx ds\right|\\
\\
\dsp \leq (1+\alpha) \frac{1}{2a} \int_{-a}^a \rho_a(s) \|\dd{U_n}{t}(s,.)\|_{\mathbb L^2(\Omega)}\|\bm \psi\|_{\mathbb L^2(\Omega)}\\\\
\dsp \leq \frac{1}{\sqrt{2a} }\|\bm \psi\|_{\mathbb L^2(\Omega)}(1+\alpha)\left( \int_{-a}^a\int_\Omega \left|\dd{U_n}{t}\right|^2dxds\right)^{\frac{1}{2}}\\
\\
\leq\dsp  \frac{1}{\sqrt{2a} }\|\bm \psi\|_{\mathbb L^2(\Omega)}(1+\alpha)\left( \int_{t_n-a}^{+\infty}\int_\Omega \left|\dd{\bm m}{t}\right|^2dxds\right)^{\frac{1}{2}}
\end{array}\end{equation*}
Since $\dsp \dd \bm m t\in L^2(\mathbb R^+;\mathbb L^2(\Omega))$, the last right hand side term tends to zero when $n$ (and so $t_n$) tends to $+\infty$. Therefore
\begin{equation*}\frac{1}{2a} \int_{-a}^a \int_\Omega \left( \dd{U_n}{t}-\alpha U_n \vect \dd{U_n}{t}\right) \bm \psi(x) \rho_a(s) dx ds\longrightarrow 0 \mbox{ when }n\longrightarrow +\infty.\end{equation*}

{\it Limit for $T_1$:} since $U_n\longrightarrow \bm u$ strongly in $\mathbb L^2([-a,a]\times \Omega)$, since $\dsp \dd{U_n}{x_i}\rightharpoonup \dd{\bm u}{x_i}$ in $\mathbb L^2(]-a,a[\times \Omega)$ weak, we obtain that
\begin{equation*}T_1\longrightarrow (1+\alpha^2)A\frac{1}{2a} \int_{-a}^a \rho_a(s) ds \int_{\Omega}
  \sum_{i=1}^3
	\left(\bm u(\bm{x}\vect\frac{\partial \bm u}{\partial x_i}(\bm{x})\right)
	\cdot\frac{\partial\bm{\psi}}{\partial x_i}(\bm{x})\ud\bm{x}.\end{equation*}

{\it Limit for $T_2$:} since $U_n$ tends to $\bm u$ strongly in $\mathbb L^2([-a,a]\times \Omega)$,
\begin{equation*}T_2\longrightarrow (1+\alpha^2)A\frac{1}{2a} \int_{-a}^a \rho_a(s) ds \int_{\Omega}
\left(\bm u(\bm{x})\vect\mathbf{K}(\bm{x})\bm u(\bm{x})\right)		
	\cdot\bm{\psi}(\bm{x})\ud\bm{x}.\end{equation*}

{\it Limit for $T_3$:} we write
\begin{equation*}T_3=-(1+\alpha^2)\int_\Omega \bm u \vect \bm h^n_a\bm \psi dx + (1+\alpha^2)\frac{1}{2a} \int_{-a}^a\int_\Omega (\bm u - U_n)h(t_n+s, x)\bm \psi(x)\rho_a(s) dx ds.\end{equation*}
We estimate the right hand side term as follows:
\begin{equation*}\begin{array}{r}
\dsp \left| \frac{1}{2a}
\frac{1}{2a} \int_{-a}^a\int_\Omega (\bm u - U_n)h(t_n+s, x)\bm \psi(x)\rho_a(s) dx ds\right|\hspace{1cm}\\
\\
\dsp \leq \|\bm \psi\|_{\mathbb L^\infty(\Omega)}\|\bm u -U_n\|_{\mathbb L^2(-a,a\times \Omega)} \|\bm h\|_{\mathbb L^2([t_n-a,t_n+a]\times \Omega)}.
\end{array}\end{equation*}
So since $U_n$ tends to $\bm u$ in $\mathbb L^2(-a,a\times \Omega)$, we obtain that
\begin{equation*}T_3\longrightarrow -(1+\alpha^2)\int_\Omega \bm{u} \vect \bm h_a\bm \psi dx.\end{equation*}

{\it Limit for $T_4$, $T_5$ and $T_6$:} since $\gamma(U_n)\longrightarrow \gamma(\bm{u})$ strongly in $\mathbb L^p([-a,a]\times \Gamma^\pm)$ for $p<+\infty$, the same occurs for $\gamma^*(U_n)$ so that we obtain:
\begin{equation*}T_4\longrightarrow -(1+\alpha^2)K_s \frac{1}{2a} \int_{-a}^a \rho_a(s)\ud s \int_{(\Gamma^\pm)}
	(\bm{\nu}\cdot\gamma \bm{u})(\gamma \bm u \vect\bm{\nu})
	\cdot\gamma\bm{\psi}(\hat{\bm{x}})\ud S(\hat{\bm{x}}),\end{equation*}
	
\begin{equation*}T_5\longrightarrow
-(1+\alpha^2)J_1\frac{1}{2a}\int_{-a}^a\rho_a(s)\ud s
\int_{(\Gamma^\pm)}
	(\gamma \bm{u} \vect\gamma^{*}\bm{u})\cdot\gamma\bm{\psi}(\hat{\bm{x}}))\ud S(\hat{\bm{x}}),	\end{equation*}
	
\begin{equation*}T_6\longrightarrow
	-2(1+\alpha^2)J_2
		\frac{1}{2a}\int_{-a}^a\rho_a(s)\ud s
\int_{(\Gamma^\pm)}
	(\gamma \bm{u}\cdot\gamma^{*}\bm{u})(\gamma \bm{u}\vect\gamma^{*}\bm{u})
	\cdot\gamma\bm{\psi}(\hat{\bm{x}}))\ud S(\hat{\bm{x}}).\end{equation*}

So we obtain that $\bm{u}$ satisfies for all $\bm \psi\in {\mathcal D}'(\overline{\Omega})$:

\begin{equation*}
\begin{array}{l}
\dsp A \int_{\Omega}
  \sum_{i=1}^3
	\left(\bm u(\bm{x}\vect\frac{\partial \bm u}{\partial x_i}(\bm{x})\right)
	\cdot\frac{\partial\bm{\psi}}{\partial x_i}(\bm{x})\ud\bm{x} 
+ A \int_{\Omega}
\left(\bm u(\bm{x})\vect\mathbf{K}(\bm{x})\bm u(\bm{x})\right)		
	\cdot\bm{\psi}(\bm{x})\ud\bm{x}\\
	\\
	\dsp 
	-\frac{2a}{\int_{-a}^a \rho_a(s)\ud s}(1+\alpha^2)\int_\Omega \bm{u} \vect \bm h_a\bm \psi dx 
-K_s  \int_{(\Gamma^\pm)}
	(\bm{\nu}\cdot\gamma \bm{u})(\gamma \bm u \vect\bm{\nu})
	\cdot\gamma\bm{\psi}(\hat{\bm{x}})\ud S(\hat{\bm{x}})\\
	\\
	\dsp 
-J_1\int_{(\Gamma^\pm)}
	(\gamma \bm{u} \vect\gamma^{*}\bm{u})\cdot\gamma\bm{\psi}(\hat{\bm{x}}))\ud S(\hat{\bm{x}})
-2J_2\int_{(\Gamma^\pm)}
	(\gamma \bm{u}\cdot\gamma^{*}\bm{u})(\gamma \bm{u}\vect\gamma^{*}\bm{u})
	\cdot\gamma\bm{\psi}(\hat{\bm{x}}))\ud S(\hat{\bm{x}})=0.
\end{array}\end{equation*}

We remark that by density, we can extend this equality for all $\bm \psi \in  \mathbb H^1(\Omega)$.

We take now the limit when $a$ tends to $+\infty$. By definition of $\rho_a$ we obtain that
\begin{equation*}\frac{2a}{\int_{-a}^a \rho_a(s)\ud s}\longrightarrow 1.\end{equation*}

Concerning $h_a$, by taking the weak limit in Estimate \eqref{estihna}, we obtain that:
\begin{equation}
\label{estiha}
\forall a\geq 1, \; 
\| \bm h_a \|_{\mathbb L^2(\mathbb R ^3)}\leq 2       \| \bm h  \|_{L^\infty(\mathbb R^+; \mathbb L^2(\mathbb R ^3))}.
\end{equation}
So by extracting a subsequence, we can assume that
\begin{equation*}\bm h_a  -\hspace{-2mm}\rightharpoonup \bm H \mbox{ in } \mathbb L^2(\mathbb R^3) \mbox{ weak  when }a\longrightarrow +\infty.
\end{equation*}

In \eqref{eq:WeakFormulationExcitation}, we take $\bm \psi(t,x)=\theta_a(t-t_n)\nabla \xi(x)$ where $\xi \in {\mathcal D}'(\mathbb R^3)$ and where
\begin{equation*}\theta_a(t)=\int_a^t \rho_a(s)\ud s.\end{equation*}
We obtain then that
\begin{equation*}\begin{array}{r}
\dsp -\mu_0\int_{-a}^a\int_{\mathbb{R}^3}(\bm{h}(t_n+s,\bm x)+\overline{U_n(s,\bm x)})\cdot
\nabla \xi (\bm x) \rho_a(s)\ud\bm{x}\ud s\\
\\
\dsp 
=\mu_0\int_{\mathbb{R}^3}(\bm{h}_0+\overline{\bm{m}_0})\cdot \nabla \xi(\bm x) \theta_a(0) \ud\bm{x}=0\end{array}
\end{equation*}
since $\div (\bm h_0 + \overline{\bm m_0})=0$

So for all $\xi \in {\mathcal D}'(\mathbb R^3)$, for all $a\geq 1$ and all $n$ great enough,
\begin{equation*}-\mu_0 \int_{\mathbb{R}^3}(\bm h_a^n(\bm x) +\frac{1}{2a}\int_{-a}^a \overline{U_n(s,\bm x)}\rho_a(s) \ud s) \cdot \nabla \xi (\bm x) \ud \bm x=0.\end{equation*}
We take the limit of this equality when $n$ tends to $+\infty$ for a fixed $a$:
\begin{equation*}-\mu_0 \int_{\mathbb{R}^3}(\bm h_a(\bm x) +\frac{1}{2a}\int_{-a}^a \rho_a(s) \ud s \overline{\bm u(\bm x)}) \cdot \nabla \xi (\bm x) \ud \bm x=0,\end{equation*}
and taking the limit when $a$ tends to $+\infty$, we get:
\begin{equation*}-\mu_0 \int_{\mathbb{R}^3}(\bm H(\bm x) + \overline{\bm u(\bm x)}) \cdot \nabla \xi (\bm x) \ud \bm x=0,\end{equation*}
that is 
\begin{equation*}\div (\bm H + \overline{\bm u})=0\mbox{ in }{\mathcal D}'(\mathbb R^3).\end{equation*}

In \eqref{eq:WeakFormulationElectric}, we take $\bm{\Theta}(t,x)=\frac{1}{2a}\rho_a(t-t_n) \xi(x)$, where $\xi \in {\mathcal D}'(\mathbb R^3)$. We obtain:
\begin{multline}\label{machin}
-\varepsilon_0\frac{1}{2a}\int_{-a}^a \int_{\mathbb{R}^3}\bm{e}(t_n+s,\bm x) \cdot
\rho_a'(s) \xi(\bm x) \ud\bm{x}\ud s
-\int_{\mathbb{R}^3}\bm{h}_a^n\cdot\Rot{\xi}\ud\bm{x}
\\+\sigma\int_{\Omega}\bm{e}_a^n\cdot \xi(\bm x)\ud \bm x    +\sigma \int_\Omega \frac{1}{2a}\int_{-a}^a \bm{f}(t_n+s, \bm x) \rho_a(s) \xi(\bm x)\ud\bm{x}\ud s
=\\=
\varepsilon_0\int_{\mathbb{R}^3}\bm{e}_0\cdot \xi(x) \rho_a(-t_n)\ud\bm{x}.
\end{multline}

For $n$ large enough, the right hand side term vanishes. We denote by $\gamma_a^n$ the term:
\begin{equation*}\gamma_a^n=-\varepsilon_0\frac{1}{2a}\int_{-a}^a \int_{\mathbb{R}^3}\bm{e}(t_n+s,\bm x) \cdot
\rho_a'(s) \xi(\bm x) \ud\bm{x}\ud s.\end{equation*}
We have:

\begin{equation*}\left\vert \gamma_a^n\right\vert \leq \frac{\varepsilon_0}{a}\| \xi\|_{L^2(\mathbb R^3)} \|\bm e \|_{L^\infty(\mathbb R^+;\mathbb L^2(\mathbb R^3))}.\end{equation*}
So for a fixed $a$, we can extract a subsequence till denoted $\gamma_a^n$ which converges to a limit $\gamma_a$ such that
\begin{equation*}\vert \gamma_a\vert\leq \frac{\varepsilon_0}{a}\| \xi\|_{L^2(\mathbb R^3)} \|\bm e \|_{L^\infty(\mathbb R^+;\mathbb L^2(\mathbb R^3))}.\end{equation*}

Moreover, 
\begin{equation*}
\begin{split}
&\phantom{\leq}\left\vert \frac{1}{2a}\int_{-a}^a \int_\Omega \bm{f}(t_n+s, \bm x)
  \rho_a(s) \xi(\bm x)\ud\bm{x}\ud s\right\vert \\
&\leq\frac{1}{2a}\left(\int_{t_n-a}^{t_n+a} \|\bm f(s,\cdot) \|^2_{\mathbb
    L^2(\Omega)}\ud s\right)^{\frac{1}{2}}\left(\int_{-a}^a
  (\rho_a(s))^2\ud s\right)^{\frac{1}{2}}\|\xi\|_{\mathbb
  L^2(\Omega)}.
\end{split}
\end{equation*}
So
\begin{equation*}\left\vert \frac{1}{2a}\int_{-a}^a \int_\Omega \bm{f}(t_n+s, \bm x) \rho_a(s) \xi(\bm x)\ud\bm{x}\ud s\right\vert \leq\frac{1}{\sqrt{2a}}\|\xi\|_{\mathbb L^2(\Omega)}
\left(\int_{t_n-a}^{+\infty} \|\bm f(s,\cdot) \|^2_{\mathbb L^2(\Omega)}\ud s\right)^{\frac{1}{2}}\end{equation*}
thus for a fixed $a$, since $\bm f\in \mathbb L^2(\R^+\times \Omega)$, this term tends to zero as $n$ tends to $+\infty$.

Therefore taking the limit when $n$ tends to $+\infty$ in \eqref{machin} we obtain:
\begin{equation*}\gamma_a -\int_{\mathbb{R}^3}\bm{h}_a\cdot\Rot{\xi}\ud\bm{x}+\sigma \int_{\Omega}\bm{e}_a\cdot \xi(\bm x)\ud \bm x =0.\end{equation*}

Taking now the limit when $a$ tends to $+\infty$ yields
\begin{equation}
\label{bidule}
-\int_{\mathbb{R}^3}\bm{H}\cdot\Rot{\xi}\ud\bm{x}+\sigma \int_{\Omega}\bm{E}\cdot \xi(\bm x)\ud \bm x =0,
\end{equation}
where $\bm E$ is a weak limit of a subsequence of $(\bm e_a)_a$.

In the same way, in \eqref{eq:WeakFormulationExcitation}, we take $\bm \psi (t,\bm x)=\rho_a(t-t_n)\xi(\bm x)$. By the same arguments, we obtain that
\begin{equation*}\int_{\mathbb R^3} \bm E \Rot \xi=0,\end{equation*}
that is $\Rot E=0$ in ${\mathcal D}'(\mathbb R^3)$.

So we remark the $\bm E$ is in $\mathbb H_{curl}(\mathbb R^3)$ and by density of ${\mathcal D}(\mathbb R^3)$ in this space, we can take $\xi = \bm E$ in \eqref{bidule}. We obtain then that
\begin{equation*}\sigma\int_\Omega \vert \bm E\vert ^2 =0.\end{equation*}
Therefore we obtain from \eqref{bidule} that 
\begin{equation*}\forall \xi \in {\mathcal D}(\mathbb R^3), \; \; \int_{\mathbb{R}^3}\bm{H}\cdot\Rot{\xi}\ud\bm{x}=0,
\end{equation*}
that is $\Rot \bm H =0$ in ${\mathcal D}'(\mathbb R^3)$.

So $\bm H$ satisfies:
\begin{equation*}\begin{array}{l}
\div(\bm H + \overline{\bm u})=0,\\
\\
\Rot \bm H =0.
\end{array}\end{equation*}
This concludes the proof of Theorem \ref{theo-omegalim}.

\section{Conclusion}
In this paper, we have proven the existence of solutions to the
Landau-Lifshitz-Maxwell system with nonlinear Neumann boundary
conditions arising from surface energies. We have also characterized
the $\omega$-limit set of those weak solutions. 

Further improvements should be possible. On the one hand, we expect that
extending these results to curved spacers should be possible. No
fundamental new idea should be necessary to carry out such an
extension of our results as long as the spacer fully separates the
domain in two. However, even in that case, the technicalities 
would lengthen the proof and the statement of the theorem as it would 
be necessary to write down geometric conditions on the spacers (the
spacer cannot share a tangent plane with the domain boundary as it
would create cusps). 

On the other hand, the construction of more regular solutions for this model remains open.

\bibliographystyle{plain}
\bibliography{LLMaxwSE-2014-02-HAL}

\end{document}